\documentclass[a4paper,11pt]{article}
\usepackage[T1]{fontenc}
\usepackage{amsmath,amsfonts,amsthm,amssymb}
\usepackage{authblk}
\usepackage[hmargin={26mm,26mm},vmargin={30mm,35mm}]{geometry}
\usepackage[latin1]{inputenc}
\usepackage{newtxtext,newtxmath}
\usepackage{numprint} 
\usepackage{paralist}
\usepackage[colorlinks,allcolors={blue},linkcolor={black}]{hyperref}
\usepackage[bold]{hhtensor}
\usepackage{comment}
\usepackage{nicefrac}
\usepackage{enumerate}
\usepackage{graphicx}
\usepackage[font=footnotesize]{caption}
\usepackage[font=footnotesize]{subcaption}
\usepackage{pgf}
\pgfkeys{/pgf/number format/.cd,relative*={3},precision=3,relative style=fixed}

\usepackage{xcolor}

\usepackage[normalem]{ulem}
\newcounter{corr}
\definecolor{violet}{rgb}{0.580,0.,0.827}
\newcommand{\corr}[3]{\typeout{Warning : a correction remains in page
\thepage}
				\stepcounter{corr}        
				{\color{blue}\ifmmode\text{\,\sout{\ensuremath{#1}}\,}\else\sout{#1}\fi}
        {\color{red}#2}
        {\color{violet} #3}}

\newcommand{\email}[1]{\href{mailto:#1}{#1}}

\newcommand{\ie}{i.e.}
\newcommand{\eg}{e.g.}

\DeclareMathOperator{\optr}{tr}
\DeclareMathOperator{\opdev}{dev}

\newcommand{\GRAD}{\vec{\nabla}}
\newcommand{\GRADs}{\GRAD_\symm}

\newcommand{\DIV}{\vec{\nabla}{\cdot}}

\newcommand{\ud}{{\rm d}}

\newcommand{\st}{\; ; \;}
\newcommand{\Id}[1][d]{\vec{I}_{#1}}

\newcommand{\norm}[2][]{\|#2\|_{#1}}
\newcommand{\seminorm}[2][]{|#2|_{#1}}

\newcommand{\symm}{{\rm s}}

\newcommand{\Real}{\mathbb{R}}
\newcommand{\Natural}{\mathbb{N}}

\newcommand{\Lvec}[2][{2}]{\textbf{L}^{#1}(#2)}
\newcommand{\Lmat}[2][{2}]{\mathbb{L}^{#1}(#2)}

\newcommand{\Hvec}[2][{1}]{\textbf{H}^{#1}(#2)}

\newcommand{\HvecGamD}[2][{1}]{\textbf{H}_{0,\Gamma_D}^{#1}(#2)}

\newcommand{\normal}{\vec{n}}

\newcommand{\ms}[1][]{\matr{\sigma}_{#1}}

\newcommand{\vu}[1][]{\vec{u}_{#1}}

\newcommand{\vv}[1][]{\vec{v}_{#1}}
\newcommand{\vw}[1][]{\vec{w}_{#1}}
\newcommand{\vf}[1][]{\vec{f}_{#1}}

\newcommand{\tph}[1][h]{\widehat{p}_h}

\newcommand{\Km}{K_{\rm m}}
\newcommand{\Kd}{K_{\rm d}}
\newcommand{\Kf}{K_{\rm f}}

\newcommand{\tF}{t_{\rm F}}
\newcommand{\dt}{{\rm d}_t}

\newtheorem{theorem}{Theorem}
\newtheorem{proposition}[theorem]{Proposition}
\newtheorem{lemma}[theorem]{Lemma}

\theoremstyle{remark}
\newtheorem{remark}[theorem]{Remark}
\theoremstyle{definition}
\newtheorem{assumption}[theorem]{Assumption}

\newcommand{\Var}{\mathbb{V}\hspace{-.2mm}{\rm ar}}
\newcommand{\Esp}{\mathbb{E}}
\newcommand{\Cov}{\mathbb{C}\hspace{-.2mm}{\rm ov}}

\title{Numerical approximation of poroelasticity with \\ random coefficients using Polynomial Chaos \\ and Hybrid High-Order methods}
\author[1]{Michele Botti\footnote{\email{michele.botti@polimi.it}}}
\author[2]{Daniele A. {Di~Pietro}\footnote{\email{daniele.di-pietro@umontpellier.fr}}}
\author[3]{Olivier {Le~Ma\^itre}\footnote{\email{olivier.le-maitre@polytechnique.edu}}}
\author[4]{Pierre Sochala\footnote{\email{p.sochala@brgm.fr}}}
\affil[1]{
  MOX, Politecnico di Milano, 20133 Milano, Italy}
\affil[2]{
  IMAG, Université de Montpellier, CNRS, 34090 Montpellier, France}
\affil[3]{
  CMAP, CNRS, INRIA, Ecole Polytechnique, 91128  Palaiseau, France}
\affil[4]{
  Bureau de Recherches Géologiques et Minières, 45060 Orléans, France}

\begin{document}

\maketitle

\begin{abstract}
In this work, we consider the Biot problem with uncertain poroelastic coefficients. The uncertainty is modelled using a finite set of parameters with prescribed probability distribution. We present the variational formulation of the stochastic partial differential system and establish its well-posedness. 
We then discuss the approximation of the parameter-dependent problem by non-intrusive techniques based on Polynomial Chaos decompositions. We specifically focus on sparse spectral projection methods, which essentially amount to performing an ensemble of deterministic model simulations to estimate the expansion coefficients. 
The deterministic solver is based on a Hybrid High-Order discretization supporting general polyhedral meshes and arbitrary approximation orders.
We numerically investigate the convergence of the probability error of the Polynomial Chaos approximation with respect to the level of the sparse grid. Finally, we assess the propagation of the input uncertainty onto the solution considering an injection-extraction problem.
\end{abstract}


\section{Introduction}

This work aims at the study poroelasticity problems where the model coefficients are uncertain. 
The interest of this type of hydro-mechanical coupled models is particularly manifest in geosciences applications~\cite{Hu.Winterfield.ea:13,Jah.Juanes:14,Minkoff.Stone.ea:03,Mehrabian.Abousleiman:15}, where subsurface fluid flows can induce a deformation of the rock matrix. 
We rely here on the linear Biot model~\cite{Biot:41,Terzaghi:43} that describes Darcean fluid flows in saturated porous media under the assumptions of small deformations and small variations of the porosity and the fluid density. 
This model depends on physical parameters that are often poorly known, justifying a stochastic description, due spatial heterogeneities, measurement inaccuracies, and sometimes the ill-posedness of inverse problems inherent to parameter estimation techniques. 
Although there is an extensive literature on poroelasticity models (we mention, in particular, the comprehensive textbooks~\cite{Coussy:04,Detournay.Cheng:93,Wang:00}) and their numerical approximation, to our knowledge few works have addressed the impact of uncertainty in Biot's formulation of hydromechanical coupling. In~\cite{Chang:85,Hong:91} the authors include uncertainty in one-dimensional consolidation analysis by incorporating heterogeneity in the consolidation coefficient. The stochastic consolidation model was extended in~\cite{Nishimura:02} to nonlinear uncertain soil parameters. Variabilities in the initial pore pressure and the heterogeneous hydraulic mobility have been considered respectively in~\cite{Darrag.Tawil:93} and~\cite{Frias.Murad.ea:04}. We also mention~\cite{Delgado.Kumar:15}, where a stochastic Galerkin approach is proposed to solve the poroelasticity equations with randomness in all material parameters and tested on a one-dimensional problem.
 
Uncertainty Quantification (UQ) methods have been developed in the last decades to take into account the effect of random model inputs on the model predictions. 
Compared to a simple deterministic simulation, advanced UQ methods can characterize an uncertain model prediction, computing its statistical moments (\eg, mean and variance), probability distributions, and performing global sensitivity analyses.
Among the techniques designed for UQ in numerical models, the stochastic spectral methods have received considerable attention. The principle of these methods is to decompose random quantities on suitable stochastic approximation bases. 
In particular, Polynomial Chaos (PC) expansions represent quantities as finites series in stochastic orthogonal polynomial bases. PC were initially introduced by Wiener~\cite{Wiener:38} and applied by Ghanem and Spanos~\cite{Ghanem:91} to solid mechanics and by Le Ma\^itre and Knio~\cite{LeMaitre:02} to fluid mechanics. PC expansion have been used to treat a large variety of problems, including elliptic models (see, \eg,~\cite{Babuska:02,Matthies.Keese:05}), flow and transport in porous media (see, \eg,~\cite{Ghanem.Dham:98, Ghanem:98}), thermal problems (see, \eg,~\cite{Hein.Kleiber:97,Liu.Hu.ea:01}), and hyperbolic systems (see, \eg,~\cite{Ern.LeMaitre.ea:10}). 
In this work, we consider a non-intrusive spectral projection method, that only requires the resolution of several deterministic problems to construct the stochastic spectral expansion of quantities of interest (``black box'' approach).
The numerical solution of the deterministic poroelasticity coupled system requires a discretization method able to
\begin{inparaenum}[(i)]
\item treat complex geometry with polyhedral meshes and nonconforming interfaces, 
\item handle possible heterogeneities of the poromechanical parameters, and 
\item prevent localized pressure oscillations arising in the case of low-permeable and low-compressible porous media. 
\end{inparaenum}
We choose to discretize the poroelasticity problem using the fully coupled method developed in~\cite{Boffi.Botti.Di-Pietro:16}, which meets all of these requirements. Therein, the elasticity operator is discretized using the Hybrid High-Order method of~\cite{Di-Pietro.Ern:15} (c.f. also~\cite{Botti.Di-Pietro.Sochala:17,Di-Pietro.Drouniou:15}), while the Darcy operator relies on the Symmetric Weighted Interior Penalty method of~\cite{Di-Pietro.Ern.ea:08}.

The contribution of this work is threefold.
First, a probabilistic framework is introduced to study the Biot model with uncertain coefficients. 
Special attention is given to defining an almost surely physically admissible set of stochastic poroelastic parameters by taking into account the dependences between the poroelastic coefficients.
Second, the well-posedness of the stochastic Biot model is proven at the continuous level.
Third, a non-intrusive PC approach is implemented in order to investigate the effect of the random poroelastic coefficients on the displacement and the pressure fields.

The material is organized as follows.
In Section~\ref{sec:model} we present the deterministic poroelasticity model and identify the relations among the poromechanical coefficients allowing to express the constrained specific storage coefficient as a function of the other parameters and primary variables.  
In Section~\ref{sec:Stochastic.Biot} we present the linear poroelasticity problem with random coefficients in strong and weak form. We then give the probabilistic assumptions on the random model coefficients yielding the well-posedness of the stochastic variational problem. 
In Section~\ref{sec:prob.setting} we propose an uncertainty model for the poroelastic coefficients. 
The sparse Pseudo-Spectral Projection (PSP) method is used to compute the PC expansions of the stochastic solutions from a finite set of deterministic problems solved using the method of \cite{Boffi.Botti.Di-Pietro:16}. Finally, we present a complete panel of numerical results in Section~\ref{sec:tests} to illustrate the performance of the method.


\section{The Biot model} \label{sec:model} 
In this section we introduce the model problem and present some preliminary relations among the coefficients characterizing the porous medium.
This linear poroelasticity model, usually referred to as the Biot model, consists of two coupled governing equations, one describing the mass balance of the fluid and the other expressing the mechanical equilibrium of the porous medium. 

\subsection{Fluid mass balance} 
We consider a spatial domain in dimension $d\ge 1$.
The fluid mass conservation law in a \textit{fully saturated} porous medium reads
\begin{equation}\label{eq:Mass_conservation0}
\dt(\rho_{\rm f}\varphi) + \DIV(\rho_{\rm f}\varphi\vv) = \rho_{\rm f}g,
\end{equation}
where $\dt$ denotes the time derivative, $\rho_{\rm f}\ [{\rm kg}/\rm{m}^3]$ is the fluid density, $\varphi\ [-]$ the soil porosity, $\vv\ [\rm{m}\rm{s}^{-1}]$ the velocity field, and $g \ [\rm{s}^{-1}]$ the fluid source. 
The equation of state for \textit{slightly compressible} fluids (cf.~\cite[Section 3.1]{Coussy:04} and~\cite{Zimmerman:00}) is
\begin{equation}\label{eq:density}
\dt\rho_{\rm f}=\frac{\rho_{\rm f}}{\Kf}\dt p,
\end{equation} 
with $p\ [{\rm Pa}]$ the pore pressure and $\Kf>0 \ [{\rm Pa}]$ the fluid bulk modulus. Under the assumptions of \textit{isotropic and isothermal} conditions, \textit{infinitesimal strains}, and \textit{small relative variations of porosity}, following~\cite[Section 4.1]{Coussy:04}, the change in the porosity caused by the fluid pressure and the mechanical displacement field $\vu\ [\rm{m}]$ is 
\begin{equation}\label{eq:porosity}
\dt\varphi=\frac1M\dt p+\alpha\dt(\DIV\vu),
\end{equation}
where $\DIV\vu$ is the divergence of the displacement field. 
We rely on the definitions of the Biot--Willis coefficient $\alpha\in(0,1]\ [-]$ and the Biot tangent modulus $M>0 \ [{\rm Pa}^{-1}]$ given in~\cite{Coussy:04,Detournay.Cheng:93}: 
\begin{equation} \label{eq:Biot_coeffs}
  \alpha\coloneq1-\frac{\Kd}{\Km},
  \qquad
  M\coloneq\frac{\Km}{\alpha-\varphi}, 
\end{equation}
with $\Kd, \Km>0\ [{\rm Pa}]$ denoting the bulk moduli of the drained medium and the solid matrix, respectively.
The coefficient $\alpha$ quantifies the amount of fluid that can be forced into the medium by a variation of the pore volume at a constant fluid pressure, while $M$ measures the amount of fluid that can be forced into the medium by pressure increments due to the compressibility of the structure. The case of a solid matrix with incompressible grains ($\Km\to +\infty$) corresponds to the limit value $\nicefrac1M=0$. From~\eqref{eq:density} and~\eqref{eq:porosity}, it follows that the variation of fluid content in the medium is given by
$$
\dt(\rho_{\rm f}\varphi)= \varphi\dt\rho_{\rm f}+\rho_{\rm f}\dt\varphi 
=\rho_{\rm f}\left(c_0\dt p+\alpha\dt(\DIV\vu)\right),
$$
where, according to~\eqref{eq:density},~\eqref{eq:porosity}, and~\eqref{eq:Biot_coeffs} the constrained specific storage coefficient is defined by 
\begin{equation}
  \label{eq:def_c0}
  c_0\coloneq\frac{\alpha-\varphi}{\Km}+\frac{\varphi}{\Kf}.
\end{equation} 
Using~\eqref{eq:Mass_conservation0} and assuming that the fluid density $\rho_{\rm f}$ is \textit{uniform} in the medium, we obtain
\begin{equation}\label{eq:Mass_conservation1}
c_0\dt p+\alpha\dt(\DIV\vu) + \DIV(\varphi\vv) = g.
\end{equation}

The fluid velocity $\vv$ is related to the pore pressure through the well-known Darcy law (see for instance~\cite[Section 4.1.2]{Detournay.Cheng:93}). Consistent with the small perturbations hypothesis adopted here,  Darcy's law can be considered in its simplest form
$$
\varphi\vv=-\frac{\mathbb{K}}{\mu_{\rm f}}\GRAD p,
$$
where $\mathbb{K}\ [\rm{m}^2]$ is the tensor-valued intrinsic permeability of the medium and $\mu_{\rm f}\ [{\rm Pa}\cdot \rm{s}]$ is the fluid viscosity. Thus, introducing the hydraulic mobility $\vec\kappa\coloneq\nicefrac{\mathbb{K}}{\mu_{\rm f}}$, the mass conservation equation~\eqref{eq:Mass_conservation1} becomes
\begin{equation}\label{eq:Biot1}
c_0\dt p + \alpha\dt(\DIV\vu)- \DIV(\vec\kappa\GRAD p) = g.
\end{equation}
For the sake of simplicity, we also assume in what follows that the mobility field is isotropic and strictly positive, namely $\vec\kappa = \kappa\Id$, with $0<\kappa<+\infty$ and $\Id$ the identity matrix.

\subsection{Momentum balance}
The momentum conservation equation under the \textit{quasi-static} assumption, namely when the inertia effects in the elastic structure are negligible, reads 
\begin{equation}
  \label{eq:momentum.balance}
-\DIV\widetilde\ms=\vf,
\end{equation}
where $\widetilde\ms \ [{\rm Pa}]$ is the total stress tensor and $\vf\ [{\rm Pa}/\rm{m}]$ is the loading force (\eg, gravity). According to Terzaghi's decomposition~\cite{Terzaghi:43}, the stress tensor is modeled as the sum of an effective term (mechanical effect) and a pressure term (fluid effect), \ie,
\begin{equation}
\label{eq:stress.decompose}
\widetilde\ms = \ms(\GRADs\vu)-\alpha p\Id.
\end{equation}
The symmetric part of the gradient of the displacement field $\GRADs\vu$ measures the strain accordingly to the small deformation hypothesis. In the context of \textit{linear isotropic} poroelasticity, the mechanical behavior of the soil is described through the Cauchy strain-stress relation defined, for all deformation tensors $\vec\epsilon\in\Real^{d\times d}_{\rm sym}$, by
\begin{equation}
  \label{eq:lin.Cauchy}
  \ms(\vec{\epsilon})=2\mu\vec\epsilon+\lambda\optr(\vec\epsilon)\Id
  =2\mu \ \vec{\opdev}(\vec\epsilon)+K\optr(\vec\epsilon)\Id,
\end{equation}
where $\mu>0 \ [{\rm Pa}]$ and $\lambda>0 \ [{\rm Pa}]$ are Lam\'e's coefficients, $K\coloneq\nicefrac{2\mu}{d}+\lambda \ [{\rm Pa}]$ is the bulk modulus, and 
$$
\optr(\vec\epsilon)\coloneq \sum_{i=1}^d \epsilon_{ii}, \quad
\vec{\opdev}(\vec\epsilon)\coloneq \vec\epsilon-\frac{\optr(\vec\epsilon)\Id}{d},
$$ 
are the trace and deviator operator, respectively. Incidentally, as noted in~\cite{Bemer.Bouteca:01,Biot:73}, physical and experimental investigations suggest that the mechanical behavior of porous solids may be nonlinear. More general stress-strain relations could be used in place of the linear law~\eqref{eq:lin.Cauchy}, as done in~\cite{Botti.Di-Pietro.Sochala:17,Botti.Di-Pietro.Sochala:18}. 
Inserting~\eqref{eq:lin.Cauchy} and~\eqref{eq:stress.decompose} into~\eqref{eq:momentum.balance}, leads to
\begin{equation}\label{eq:Biot2}
-\DIV\left({2\mu\GRADs\vu + (\lambda\DIV\vu-\alpha p) \Id}\right) = \vf.
\end{equation}

\subsection{Constrained specific storage coefficient}\label{sec:storage_uncertain}
In order to perturb the Biot model coefficients, we need to investigate the dependences between the poroelastic coefficients. 
We propose here to express the specific storage coefficient $c_0$ as a function of other physical parameters.
For that purpose, we introduce the Gassmann equation (cf.~\cite{Gassmann:51,Mehrabian.Abousleiman:15}), that relates the bulk moduli $K$ and $\Kd$ to the Biot--Willis and storage coefficients $\alpha,c_0$:
\begin{equation}
  \label{eq:Gassmann}
  \alpha^2= {c_0}(K -\Kd).
\end{equation}
Using~\eqref{eq:Biot_coeffs} to express the drained bulk modulus $\Kd$ as a function of $\alpha$ and $\Km$, equation~\eqref{eq:Gassmann} gives
\begin{equation}
  \label{eq:lbnd_c0}
  \Km = \frac{K c_0 - \alpha^2}{c_0 (1-\alpha)}.
\end{equation}
Plugging the previous equation into~\eqref{eq:def_c0} and rearranging yields the following quadratic equation for $c_0$:
\begin{equation}
  \label{eq:2order_c0}
  K c_0^2-\left(\alpha+\alpha\varphi+\frac{\varphi K}{\Kf}-\varphi\right)c_0+\frac{\alpha^2\varphi}{\Kf}=0.
\end{equation}
The parameters $\varphi,\Kf,K$, and $\alpha$ can then be used to define $c_0$ through the previous equation. Some conditions need to be prescribed in order to avoid non-physical solutions $c_0$. Owing to the definition of the Biot--Willis coefficient, it holds that $0\le\varphi\le\alpha\le 1$. 
As observed in~\cite[Section 3]{Zimmerman:00}, a stricter lower bound can be imposed, namely
\begin{equation}
  \label{eq:bnd.Zimmerman}
  \frac{3\varphi}{2+\varphi}\le\alpha.
\end{equation}

\begin{lemma}[Existence]
  Let $\Kf,K>0$, and $\varphi,\alpha$ satisfying condition~\eqref{eq:bnd.Zimmerman}. Then, there 
  exists $c_0\in\Real^+$ solution of~\eqref{eq:2order_c0}.
\end{lemma}
\begin{proof}
We prove existence by assessing the positivity of the discriminant $\mathcal{D}$ associated to~\eqref{eq:2order_c0}. Computing $\mathcal{D}$ and rearranging, leads to
$$
\begin{aligned}
 \mathcal{D} &= \left(\alpha+\alpha\varphi+\frac{\varphi K}{\Kf}-\varphi\right)^2 - \frac{4\alpha^2\varphi K}{\Kf}
 \\ 
 &=\varphi^2\left(\nicefrac{K}{\Kf}-1\right)^2 
 + 2\alpha\varphi\left(1+\varphi-2\alpha\right)\left(\nicefrac{K}{\Kf} -1\right) + \alpha^2(1-\varphi)^2
 \\
 &=\left[\varphi\left(\nicefrac{K}{\Kf}-1\right) + \alpha(1+\varphi-2\alpha)\right]^2 
 + \alpha^2(1-\varphi)^2-\alpha^2(1+\varphi-2\alpha)^2
 \\
 &=\left[\varphi\left(\nicefrac{K}{\Kf}-1\right) + \alpha(1+\varphi-2\alpha)\right]^2 
 + 4\alpha^2(1-\alpha)(\alpha-\varphi).
\end{aligned}
$$
Since $0\le\varphi<\alpha\le1$ owing to~\eqref{eq:bnd.Zimmerman}, the second term in the previous sum is positive and, as a result, $\mathcal{D}\ge0$.
\end{proof}
\begin{remark}
The previous lemma yields the existence of two real solutions $c_0^-\le c_0^+$. We consider $c_0^+$ as the unique solution to~\eqref{eq:2order_c0} because, for admissible values of $(K,\Kf,\varphi,\alpha)$, we might have $c_0^-< \nicefrac{\varphi}{\Kf}$, violating~\eqref{eq:def_c0}. For instance, this is the case if $\alpha=\nicefrac{2\varphi}{(1+\varphi)}$ for any $\phi\le \nicefrac34$ and $K,\Kf>0$. 
\end{remark}


\section{The Biot problem with random coefficients}\label{sec:Stochastic.Biot}

In this section, we introduce the extension of the linear poroelasticity problem in the case of uncertain poroelastic coefficients. We discuss the assumptions on the random coefficients, derive a weak formulation of the resulting stochastic problem, and prove its well-posedness.

\subsection{Strong and weak formulations}\label{sec:weak_form}
Let $n\in\Natural$ and denote by $X\subset\Real^n$ a measurable set. Spaces of functions, vector fields, and tensor fields defined over $X$ are respectively denoted by italic capital, boldface Roman capital, and special Roman capital letters. Thus, for example, $L^2(X)$, $\Lvec{X}$, $\Lmat{X}$ denote the spaces of square integrable functions, vector fields, and tensor fields over $X$, respectively.
For any $m\in\Natural$, we denote by $H^m(X)$ the usual Sobolev space of functions that have weak partial derivatives of order up to $m$ in $L^2(X)$, with the convention that $H^0(X)\coloneq L^2(X)$, while $C^m(X)$ and $C_{\rm c}^\infty(X)$ denote, respectively, the usual spaces of $m$-times continuously differentiable functions and infinitely continuously differentiable functions with compact support on $X$. 

We introduce an abstract probability space $(\Theta,\mathcal{B},\mathcal{P})$, where $\Theta$ is the set of possible outcomes, $\mathcal{B}$ a $\sigma$-algebra of events, and $\mathcal{P}:\mathcal{B}\to[0,1]$ a probability measure. For any random variable $h:\Theta\to\Real$ defined on the abstract probability space, the expectation of $h$ is 
$$
\mathbb{E}(h)\coloneq\int_\Theta h(\theta) \ud\mathcal{P}(\theta).
$$ 
We assume hereafter that all random quantities are second-order ones, namely they belong to 
$$
L^2_\mathcal{P}(\Theta)\coloneq \left\{ h:\Theta\to\Real \st \mathbb{E}(h^2)<+\infty\right\}.
$$

The Biot problem consists in finding a vector-valued displacement field $\vu$ and a scalar-valued pressure field $p$ solutions of~\eqref{eq:Biot1} and~\eqref{eq:Biot2}. We now consider the stochastic version corresponding to the case of random poroelastic coefficients. We assume that $D\subset\mathbb{R}^d$, $d\in\{2,3\}$, is a (deterministic) bounded connected polyhedral domain with boundary $\partial D$ and outward normal $\normal$. 
In order to close the problem, we enforce boundary conditions corresponding to a medium that is clamped on $\Gamma_D\subset\partial D$, traction-free on $\Gamma_N\coloneq\partial{D}\setminus\Gamma_D$, permeable with free drainage on $\Gamma_d\subset\partial D$, and impermeable on $\Gamma_n\coloneq\partial{D}\setminus\Gamma_d$, as well as a deterministic initial condition prescribing the initial fluid content $\phi_0$. For a given finite time $t_F>0$, the resulting problem is given by
\begin{subequations}\label{eq:SysBiot}
\begin{alignat}{7}
    \label{eq:SysBiot1}
    -\DIV\ms(\vec{x},\theta,t)
    +\GRAD(\alpha(\vec{x},\theta) p(\vec{x},\theta,t)) &=\vf(\vec{x},t), 
    &\qquad& (\vec{x},\theta,t)\in D\times\Theta\times (0,t_F],
    \\
    \label{eq:SysBiot2}
    \dt \phi(\vec{x},\theta,t) - \DIV(\vec\kappa(\vec{x},\theta)\GRAD p(\vec{x},\theta,t)) &= g(\vec{x},t), 
    &\qquad& (\vec{x},\theta,t)\in D\times\Theta\times (0,t_F],
    \\
    \label{eq:SysBiot:bcD.u}
    \vu(\vec{x},\theta,t) &= \vec{0}, &\qquad&(\vec{x},\theta,t)\in \Gamma_{D}\times\Theta\times (0,t_F],
    \\
    \label{eq:SysBiot:bcN.u}
    \ms(\vec{x},\theta,t)\normal+\alpha(\vec{x},\theta) p(\vec{x},\theta,t))\normal, 
    &= \vec{0} &\qquad& (\vec{x},\theta,t)\in \Gamma_{N}\times\Theta\times (0,t_F],
    \\
    \label{eq:SysBiot:bcD.p}
    p(\vec{x},\theta,t) &= 0, &\qquad&(\vec{x},\theta,t)\in \Gamma_{d}\times\Theta\times (0,t_F],
    \\
    \label{eq:SysBiot:bcN.p}
    \vec\kappa(\vec{x},\theta)\GRAD p(\vec{x},\theta,t)\cdot\normal &= 0, 
    &\qquad& (\vec{x},\theta,t)\in \Gamma_{n}\times\Theta\times (0,t_F],
    \\
    \label{eq:SysBiot:initial}
    \phi(\vec{x},\theta,0)&=\phi_0(\vec{x}), &\qquad& (\vec{x},\theta)\in D\times\Theta,
\end{alignat}
where the stress tensor in~\eqref{eq:SysBiot1} and~\eqref{eq:SysBiot:bcN.u}, and the fluid content in~\eqref{eq:SysBiot2} and~\eqref{eq:SysBiot:initial}, are defined, for all $(\vec{x},\theta,t)\in D\times\Theta\times (0,t_F]$, such that
\begin{align}
    \label{eq:SysBiot:stress}
    \ms(\vec{x},\theta,t) &= 2\mu(\vec{x},\theta)\GRADs\vu(\vec{x},\theta,t)
    +\lambda(\vec{x},\theta)(\DIV\vu(\vec{x},\theta,t))\Id,
    \\
    \label{eq:SysBiot:fluid_content}
    \phi(\vec{x},\theta,t) &= 
    c_0(\vec{x},\theta) p(\vec{x},\theta,t) + \alpha(\vec{x},\theta)\DIV\vu(\vec{x},\theta,t).
\end{align}
If $\Gamma_N=\Gamma_d=\emptyset$ and $c_0=0$, owing to~\eqref{eq:SysBiot2} and the homogeneous Neumann condition~\eqref{eq:SysBiot:bcN.p}, we need the following compatibility conditions on $g$ and $\phi_0$ and zero-average constraint on $p$:
\begin{equation}\label{eq:compatibility2}
  \int_D \phi_0(\cdot) = 0, \qquad\quad\; \int_D g(\cdot,t) = 0,\quad\text{ and }\quad\int_D p(\cdot,\cdot,t) = 0
  \qquad\forall t\in(0,t_F).
\end{equation}
\end{subequations}
For the sake of simplicity, we do not consider the case of $\Gamma_D$ with zero $(d-1)$-dimensional Hausdorff measure. This situation simply requires an additional compatibility condition on $\vf$ and the prescription of the rigid-body motions of the medium. We remark that inhomogeneous and possibly random boundary conditions, as well as random loading, can also be considered up to minor modifications.

Before giving the variational formulation of problem~\eqref{eq:SysBiot}, we introduce some notations. For a given vector space $V(D)$ of real-valued functions on $D$, equipped with the norm $\norm[V(D)]{\cdot}$, we define the Bochner space of second-order random fields by
$$
L^{2}_\mathcal{P}(\Theta;V(D))\coloneq \left\{v:D\times\Theta\to\Real\st
\norm[L^{2}_\mathcal{P}(\Theta;V(D))]{v}\coloneq\mathbb{E}\left(\norm[V(D)]{v}^2\right)^{\frac12}<+\infty\right\}.
$$
Since the Biot problem is unsteady, we also need to introduce, for a given vector space $W$, the Bochner space $L^{2}((0,\tF);W)$ of square integrable $W$-valued functions of the time interval $(0,\tF)$. Similarly, the Sobolev space $H^{1}((0,\tF);W)\subset L^{2}((0,\tF);W)$ is spanned by $W$-valued functions having square integrable first-order weak derivative.
In what follows, for all $(\vec{x},\theta,t)\in D\times\Theta\times(0,t_F]$, $v\in L^{2}_\mathcal{P}(\Theta;V(D))$, and 
$w\in L^{2}((0,\tF);L^{2}_\mathcal{P}(\Theta;V(D)))$, we use $v(\vec{x},\theta)$ and $w(\vec{x},\theta,t)$ as shorthand notations for $(v(\theta))(\vec{x})$ and $(w(t))(\vec{x},\theta)$, respectively. 

At each time $t\in (0,\tF]$, the natural functional spaces for the random displacement field $\vu(t):D\times\Theta\to\Real^d$ and pressure field $p(t):D\times\Theta\to\Real$ taking into account the boundary conditions~\eqref{eq:SysBiot:bcD.u}--\eqref{eq:SysBiot:bcN.p} are, respectively, 
$$
  \begin{aligned}
  \vec{U}&\coloneq L^2_\mathcal{P}(\Theta; \HvecGamD{D}),
  \\
  P&\coloneq\begin{cases}
  L^2_\mathcal{P}(\Theta; {H}^1(D)\cap L_0^2(D))  & 
  \quad\text{if $\Gamma_d=\Gamma_N=\emptyset$ and $c_0=0$,} \\
  L^2_\mathcal{P}(\Theta; {H}^1_{0,\Gamma_d}(D)) & \quad\text{otherwise,} 
  \end{cases}
  \end{aligned}
$$
with $\,\HvecGamD{D}\coloneq\left\{\vv\in\Hvec{D}:{\vv}_{|\Gamma_D}=\vec0\right\}$, 
$\,{H}^1_{0,\Gamma_d}(D)\coloneq\left\{q\in H^1(D): q_{|\Gamma_d}=0\right\}\,$, and
$$
L_0^2(D)\coloneq\left\{q\in L^2(D):\int_D q=0\right\}.
$$
We consider the following weak formulation of problem~\eqref{eq:SysBiot}:
For a loading $\vf\in L^2((0,\tF);\Lvec{D})$, a fluid source $g\in L^2((0,\tF);L^2(D))$, and an initial datum $\phi_0\in L^2(D)$ that verify~\eqref{eq:compatibility2}, find $\vu\in L^2((0,\tF);\vec{U})$ and $p\in L^2((0,\tF);P)$ such that, for all $\vv\in\vec{U}$, all $q\in P$, and all $\psi\in C_{\rm c}^\infty((0,\tF))$
\begin{subequations}
  \label{eq:weak_form}
  \begin{align}
    \label{eq:weak_form.mech}
    \int_0^{\tF}a(\vu(t),\vv)\ \psi(t) \ud t+\int_0^{\tF}\hspace{-1mm}b(\vv,p(t))\ \psi(t)\ud t
    &=\int_0^{\tF}\hspace{-2mm}\mathbb{E}\left((\vf(t),\vv)_D\right)\psi(t)\ud t,
    \\
    \label{eq:weak_form.fluid}
    \hspace{-1mm}\int_0^{\tF}\hspace{-1mm}\left[b(\vu(t),q)-c(p(t),q)\right] \dt\psi(t) \ud t 
    +\int_0^{\tF}\hspace{-1mm}d(p(t),q)\ \psi(t) \ud t
    &=\int_0^{\tF}\hspace{-2mm}\mathbb{E}\left( (g(t),q)_D\right) \psi(t) \ud t, 
    \\
    \label{eq:weak_form.initial}
    c(p(0),q)-b(\vu(0),q)&= \mathbb{E}\left((\phi_0,q)_D \right),
  \end{align}
\end{subequations}
where $(\cdot,\cdot)_D$ denotes the usual inner product in $L^2(D)$ and the bilinear forms $a:\vec{U}\times\vec{U}\to\Real$, $b:\vec{U}\times P\to\Real$, and $c:P\times P\to\Real$ are defined such that, for all $\vv,\vw\in\vec{U}$ and all $q,r\in P$,
$$
\begin{aligned}
  a(\vv,\vw)&\coloneq \mathbb{E}\left(\int_D 
  \left(2\mu(\vec{x},\cdot)\GRADs\vv(\vec{x},\cdot):\GRADs\vw(\vec{x},\cdot) +
  \lambda(\vec{x},\cdot)\DIV\vv(\vec{x},\cdot)\DIV\vw(\vec{x},\cdot)\right) \ud\vec{x}\right),
  \\
  b(\vv,q)&\coloneq \mathbb{E}\left(-\int_D \alpha(\vec{x},\cdot)\DIV\vv(\vec{x},\cdot) q(\vec{x},\cdot)\ 
  \ud\vec{x}\right),
  \\
  c(q,r)&\coloneq \mathbb{E}\left(\int_D c_0(\vec{x},\cdot) q(\vec{x},\cdot) r(\vec{x},\cdot)\ 
  \ud\vec{x}\right),
  \\
  d(q,r)&\coloneq \mathbb{E}\left(\int_D\kappa(\vec{x},\cdot)\GRAD r(\vec{x},\cdot)\cdot\GRAD q(\vec{x},\cdot)\ 
  \ud\vec{x}\right).
\end{aligned}
$$
Above, we have introduced the Frobenius product such that, for all $\matr{\tau},\matr{\eta}\in\Real^{d\times d}$, $\matr{\tau}:\matr{\eta}\coloneq\sum_{1\le i,j\le d}\tau_{ij}\eta_{ij}$ with corresponding norm such that, for all $\matr{\tau}\in\Real^{d\times d}$, $\seminorm[d\times d]{\matr{\tau}}\coloneq(\matr{\tau}:\matr{\tau})^{\nicefrac12}$.

\subsection{Well-posedness}\label{sec:wf_well-posedness}

The aim of this section is to infer a stability estimate on the displacement and pressure $(\vu,p)\in L^2((0,\tF);\vec{U})\times L^2((0,\tF);P)$ solving~\eqref{eq:weak_form} yielding, in particular, the well-posedness of the weak problem.
The existence and uniqueness of a solution to the deterministic Biot problem has been studied in \cite{Showalter:00,Zenisek:84}. The results therein establish the existence, for almost every (a.e.) $\theta\in\Theta$, of a solution $(\vu(\theta),p(\theta))$ to~\eqref{eq:SysBiot}. Following the argument of \cite[Theorem 2.2]{Charrier:13}, we note that the mapping $(\vu,p):\Theta\to L^2((0,\tF); \HvecGamD{D})\times L^2((0,\tF); H^1(D))$ is measurable because it is a continuous function of the input coefficients $\mu,\lambda,\alpha,c_0,\kappa\in L^2_\mathcal{P}(\Theta)$.
Therefore, proving that $\theta\mapsto(\vu(\theta),p(\theta))$ is bounded in $L^2_\mathcal{P}(\Theta)$ gives the existence of a solution to~\eqref{eq:weak_form}. This result is established by Proposition \ref{pro:aprori}. 

We assume some additional conditions on the input random data that will be needed in the proofs. First, we recall that, for all $\vec{x}\in D$ and all $\theta\in\Theta$, the coupling coefficient $\alpha(\vec{x},\theta)\in (0,1]$ satisfies the lower bound~\eqref{eq:bnd.Zimmerman}. We assume that the reference porosity of the medium is strictly positive (otherwise the Biot problem would reduce to a decoupled one), so that we have $0<\underline\alpha\coloneq\frac{3\varphi}{2+\varphi}<\alpha(\vec{x},\theta)$. We also assume that Lam\'e's  coefficients satisfy the following:
\begin{assumption}[Elastic moduli]\label{ass:shear_mod}
  The shear modulus $\mu\in L^\infty(D\times\Theta)$ is uniformly bounded away from zero, \ie~there exists a positive constant $\underline\mu$ such that 
  \begin{equation}\label{eq:ass.shear_mod}
    0<\underline\mu \le\mu(\vec{x},\theta), \qquad {\rm a.e.} \,\text{ in } D\times\Theta.
  \end{equation}
  The bulk modulus $K\in L^\infty(D\times\Theta)$ is uniformly bounded from above, \ie~there exists a constant $\overline{K}$ such that
  \begin{equation}\label{eq:ass.dilat_mod}
  0<\frac{2\mu(\vec{x},\theta)+d\lambda(\vec{x},\theta)}{d}=K(\vec{x},\theta)\le\overline{K}<+\infty, 
  \qquad {\rm a.e.} \,\text{ in } D\times\Theta.
  \end{equation}
\end{assumption}
The next result establishes the coercivity and the inf-sup condition of the bilinear forms $a$ and $b$ in~\eqref{eq:weak_form.mech}.
\begin{lemma}[Coercivity and inf-sup]\label{lem:coer_infsup}
  Under Assumption~\ref{ass:shear_mod}, the following bounds hold:
  \begin{align}
      \label{eq:coercivity_a}
     a(\vv,\vv)&\ge 2\underline\mu \ C_{\rm K}^{-1}\norm[\vec{U}]{\vv}^2, \;\qquad\qquad\forall\vv\in\vec{U},
     \\
     \label{eq:inf-sup_b}
     \sup_{\vec0\neq\vv\in\vec{U}} \frac{b(\vv,q)}{\norm[\vec{U}]{\vv}}&\ge 
     \underline\alpha \ C_{\rm is} 
     \norm[L^2_\mathcal{P}(\Theta; L^2(D))]{q}, \qquad\forall q\in L^2_\mathcal{P}(\Theta; L^2(D)),
  \end{align}
  where $C_{\rm K}>0$ denotes the constant in Korn's first inequality, $C_{\rm is}>0$ the inf-sup constant, and 
  $\norm[\vec{U}]{\cdot}\coloneq\norm[L^2_\mathcal{P}(\Theta;{\HvecGamD{D}})]{\cdot}$. 
  In the case $\Gamma_D=\partial D$,~\eqref{eq:inf-sup_b} holds for all $q\in L^2_\mathcal{P}(\Theta;L_0^2(D))$.
\end{lemma}
\begin{proof}
The first bound directly follows from Assumption~\ref{ass:shear_mod}. Indeed, since $|\Gamma_D|_{d-1} >0$, we can apply Korn's first inequality, yielding
$$
2\underline\mu\norm[L^2_\mathcal{P}(\Theta; \Lmat{D})]{\GRAD\vv}^2
\le 2\underline\mu C_{\rm K} \mathbb{E}\left(\norm[\Lmat{D}]{\GRADs\vv}^2 \right)
\le C_{\rm K} a(\vv,\vv), \quad\;\forall\vv\in\vec{U}.
$$
To obtain~\eqref{eq:inf-sup_b}, we use \cite[Lemma 7.2]{Bespalov.Powell.ea:12} establishing the existence, for any $q\in L^2_\mathcal{P}(\Theta; L^2(D))$ 
(or $q\in L^2_\mathcal{P}(\Theta; L_0^2(D))$ in the case $\Gamma_d=\partial D$), of ${\vv}_q\in \vec{U}$ such that $\DIV{\vv}_q = q$ and $C_{\rm is}\norm[\vec{U}]{{\vv}_q}\le \norm[L^2_\mathcal{P}(\Theta; L^2(D))]{q}$, with $C_{\rm is}>0$ depending on $D$. Thus, we conclude 
$$
\sup_{\vec0\neq\vv\in\vec{U}} \frac{b(\vv,q)}{\norm[\vec{U}]{\vv}}\ge \frac{-b({\vv}_q,q)}{\norm[\vec{U}]{{\vv}_q}}
\ge \frac{\underline\alpha \norm[L^2_\mathcal{P}(\Theta; L^2(D))]{q}^2}{\norm[\vec{U}]{{\vv}_q}}
\ge \underline\alpha\ C_{\rm is}\norm[L^2_\mathcal{P}(\Theta; L^2(D))]{q}.
$$
\end{proof}

In order to prove the stability estimate of Proposition~\ref{pro:aprori} below, we infer from~\eqref{eq:ass.dilat_mod} a lower bound for the specific storage coefficient $c_0$. We rewrite~\eqref{eq:lbnd_c0}, derived from the definition of $\alpha$ in~\eqref{eq:Biot_coeffs} and the Gassmann equation~\eqref{eq:Gassmann}, in the form
$$
\frac{(1-\alpha)(\Km - K)}{\alpha} = K - \frac{\alpha}{c_0}. 
$$
Since the left-hand side of the previous relation is always non-negative (according to \cite[Chapter 4]{Coussy:04} we have $\Kd\le K\le\Km$), we infer that $\nicefrac\alpha{c_0}\le K$ and, as a result,
\begin{equation}
  \label{eq:bndc0}
c_0^{-1}(\vec{x},\theta)\le \frac{K(\vec{x},\theta)}{\alpha(\vec{x},\theta)}\le
\frac{\overline{K}}{\underline\alpha}
\qquad {\rm a.e.} \,\text{ in } D\times\Theta.
\end{equation}
\begin{proposition}[A priori estimate]\label{pro:aprori}
  Let $(\vu,p)\in L^2((0,\tF);\vec{U}\times P)$ solve~\eqref{eq:weak_form}. Then, under Assumption~\ref{ass:shear_mod}, 
  it holds
  \begin{equation}\label{eq:aprori}
  \int_0^{\tF} \left[a(\vu(t),\vu(t)) + c(p(t),p(t))\right]\ud{t}
  \le \int_0^{\tF} \left(\frac{C_{\rm K}}{2\underline\mu}\norm[\Lvec{D}]{\vf}^2
  +\frac{\overline{K}}{\underline\alpha}\norm[L^2(D)]{G}^2 \right) \ud{t}, 
  \end{equation}
  with $G:(0,\tF)\to L^2(D)$ defined by 
  \begin{equation}\label{eq:defGrhs}
    G(t)\coloneq \int_0^t g(s)\ud{s} + \phi_0.
  \end{equation}
\end{proposition}
\begin{proof} 
Since the loading $\vf$ is deterministic, from~\eqref{eq:weak_form.mech} we infer that 
$a(\vu(t),\vv) + b(\vv,p(t)) = \left(\vf(t),\mathbb{E}(\vv)\right)_D$, for ${\rm a.e.}\;t\in(0,\tF)$ and all $\vv\in\vec{U}$. Setting $\vv(\vec{x},\theta)=\vu(\vec{x},\theta,t)$ and integrating the previous relation on $(0,\tF)$, yields
\begin{equation}
  \label{eq:energy_balance1}
  \int_0^{\tF} a(\vu(t),\vu(t)) \ud{t} + \int_0^{\tF} b(\vu(t),p(t)) \ud{t} = 
  \int_0^{\tF}\left(\vf(t),\mathbb{E}(\vu(t))\right)_D \ud{t}.
\end{equation}
Adapting the argument of \cite[Chapter 4]{Botti_PhD:18} and \cite[Remark 2]{Botti.Di-Pietro.Sochala:18}, we infer that $b(\vu(t),q)-c(p(t),q)\in H^1((0,\tF))$ for all $q\in P$. 
Thus, we can integrate by parts~\eqref{eq:weak_form.fluid} and obtain 
$$
\dt\left[c(p(s),q)-b(\vu(s),q)\right] +d(p(s),q) =\left(g(s),\mathbb{E}(q)\right)_D, 
\quad \text{for }{\rm a.e.}\;s\in(0,\tF).
$$
Letting $t\in(0,\tF)$, integrating the previous identity on $(0,t)$, and then taking $q(\vec{x},\theta)=p(\vec{x},\theta,t)$, leads to
$$
c(p(t),p(t))-b(\vu(t),p(t)) + \int_0^t d(p(s),p(t)) \ud{s} 
= \int_0^t \left(g(s),\mathbb{E}(p(t))\right)_D \ud{s} + \left(\phi_0,\mathbb{E}(p(t))\right)_D.
$$
We define $z(t)=\int_0^t p(s) \ud{s}$ and observe that $\dt z(t)=p(t)$ and $z(0)=0$. Using the linearity of $d$ and the formula $\dt z^2(t) = 2 z(t)\dt z(t)$ to rewrite the third term in the left hand side of the previous relation, and recalling~\eqref{eq:defGrhs}, we get 
$$
c(p(t),p(t))-b(\vu(t),p(t)) + \frac12\dt \left[d(z(t),z(t)) \right]
= \left(G(t),\mathbb{E}(p(t))\right)_D.
$$ 
Then, integrating on $(0,\tF)$ and summing the resulting identity to~\eqref{eq:energy_balance1}, gives 
\begin{multline}
  \label{eq:energy_balance2}
  \int_0^{\tF} a(\vu(t),\vu(t))\ud{t} + \int_0^{\tF}c(p(t),p(t))\ud{t}+ \frac12 d(z(\tF),z(\tF)) = 
  \\
  \int_0^{\tF}\left(\vf(t),\mathbb{E}(\vu(t))\right)_D \ud{t}+
  \int_0^{\tF}\left(G(t),\mathbb{E}(p(t))\right)_D \ud{t}.
\end{multline}
To establish the result, we now bound the right-hand side of the previous relation. First, we observe that, owing to the Jensen inequality, for all $\vv\in\vec{U}$ it holds
$$
\norm[\Lvec{D}]{\mathbb{E}(\vv)}^2
=\int_D \mathbb{E}( \vv(\vec{x},\cdot))^2 \ud{\vec{x}}
\le \mathbb{E}\left( \int_D \vv(\vec{x},\cdot)^2\ \ud{\vec{x}} \right)
= \norm[L_\mathcal{P}^2(\Theta;\Lvec{D})]{\vv}^2
\le \norm[\vec{U}]{\vv}^2.
$$ 
Using the previous relation and the Cauchy--Schwarz and Young inequalities followed by~\eqref{eq:coercivity_a}, it is inferred that
$$
\begin{aligned}
\int_0^{\tF}\left(\vf(t),\mathbb{E}(\vu(t))\right)_D \ud{t} &\le
\int_0^{\tF} (2\underline\mu C_{\rm K}^{-1})^{\nicefrac{-1}2}\norm[\Lvec{D}]{\vf}
(2\underline\mu C_{\rm K}^{-1})^{\nicefrac12}\norm[\vec{U}]{\vu} \ud{t} 
\\
&\le \frac{C_{\rm K}}{4\underline\mu}\int_0^{\tF} \norm[\Lvec{D}]{\vf}^2 \ud{t} 
+ \frac12\int_0^{\tF} a(\vu(t),\vu(t))\ud{t}.
\end{aligned}
$$
Proceeding similarly for the second term in the right-hand side of~\eqref{eq:energy_balance2} and recalling the lower bound~\eqref{eq:bndc0}, one has
$$
\begin{aligned}
  \int_0^{\tF}\left(G(t),\mathbb{E}(p(t))\right)_D \ud{t} &\le
  \int_0^{\tF} \norm[L_\mathcal{P}^2(\Theta;L^2(D))]{c_0^{\nicefrac{-1}{2}} G}
  \norm[L_\mathcal{P}^2(\Theta;L^2(D))]{c_0^{\nicefrac12} p} \ud{t}
  \\
  &\le\frac{\overline{K}}{\underline\alpha}\int_0^{\tF} \norm[L^2(D)]{G}^2 \ud{t} 
  + \frac12\int_0^{\tF} c(p(t),p(t))\ud{t}.
\end{aligned}
$$
Plugging the two previous estimates into~\eqref{eq:energy_balance2}, multiplying by a factor $2$, and using $d(z(\tF),z(\tF))\ge0$, yields the conclusion.
\end{proof}
Some remarks are in order. 
\begin{remark}[Inf-sup condition]
We observe that it is possible to derive the previous a priori estimate without~\eqref{eq:bndc0}. Indeed, the inf-sup condition~\eqref{eq:inf-sup_b} together with~\eqref{eq:weak_form.mech} and \eqref{eq:ass.dilat_mod} allow to bound the second term in the right-hand side of~\eqref{eq:energy_balance2} using the following: for a.e. $t\in (0,\tF)$,
\begin{equation*}
\begin{aligned}
\norm[L^2_\mathcal{P}(\Theta; L^2(D))]{p(t)} &\le (\underline\alpha  C_{\rm is})^{-1} 
\sup_{\vec0\neq\vv\in\vec{U}} \frac{b(\vv,p(t))}{\norm[\vec{U}]{\vv}} = (\underline\alpha  C_{\rm is})^{-1}
\sup_{\vec0\neq\vv\in\vec{U}}\frac{(\vf(t),\mathbb{E}(\vv))_D-a(\vu(t),\vv)}{\norm[\vec{U}]{\vv}}
\\ &\le
\frac1{\underline\alpha  C_{\rm is}}\norm[\Lvec{D}]{\vf(t)}+
\frac{\overline{K}^{\nicefrac12}}{\underline\alpha C_{\rm is}} 
a(\vu(t),\vu(t))^{\nicefrac12}.
\end{aligned}
\end{equation*}
The resulting stability estimate would have, compared to~\eqref{eq:aprori}, an additional dependence on $\underline\alpha C_{\rm is}$. 
\end{remark}
\begin{remark}[Data regularity]
In order to prove the a priori bound \eqref{eq:aprori}, no additional time regularity assumption on the loading term 
$\vf$ and the mass source $g$ are needed. However, under the additional requirements $\vf\in C^1([0,\tF];\Lvec{\Omega})$ and $g\in C^1([0,\tF];L^2(\Omega))$, a stronger version of \eqref{eq:aprori} can be inferred using the argument of \cite{Marciniak-Czochra:15} and \cite{Owczarek:10}.
\end{remark}
\begin{remark}[Quasi-incompressible media]
In order to prove the stability estimate~\eqref{eq:aprori}, no additional assumption on the mobility $\kappa:D\times\Theta\to\Real^+$ is required. Thus, Proposition~\ref{pro:aprori} can handle the case of locally poorly permeable media (\ie~$\kappa\ll1$). However, assuming~\eqref{eq:ass.dilat_mod} does not allow to consider quasi-incompressible materials for which $\lambda\gg1$. To obtain a robust estimate in the case of $\lambda$ unbounded, we can proceed as in \cite[Theorem 1]{Kolesov.Vabishchevich.ea:14}. Assuming $\vf\in H^1((0,\tF);\Lvec{D})$ and $\kappa$ uniformly bounded away from zero, the Darcy term gives a $L^2((0,\tF);P)$ estimate of the pressure and, as a consequence,~\eqref{eq:ass.dilat_mod} and~\eqref{eq:bndc0} are not needed. We remark that, in a medium featuring very low permeability, an incompressible fluid cannot flow unless the material is compressible, namely the two limit cases $\kappa\ll1$ and $\lambda\gg1$ cannot occur simultaneously. Therefore, the stochastic poroelasticity problem need to be dealt with uncertainty models preventing this situation.
\end{remark} 
\begin{remark}[Lam\'e's coefficients]
The well-posedness of problem~\eqref{eq:weak_form} holds under weaker assumption on $\mu$ and $\lambda$. More precisely, one can assume instead of~\eqref{eq:ass.shear_mod} and~\eqref{eq:ass.dilat_mod}, that $\mathcal{P}$-a.e. in $\Theta$ it holds
\begin{equation*}
0<\underline\mu(\theta)\le\mu(\vec{x},\theta), \qquad
0\le \frac{2\mu(\vec{x},\theta)+d\lambda(\vec{x},\theta)}{d}\le\overline{K}(\theta) 
\quad {\rm a.e.} \,\text{ in } D,
\end{equation*} 
where $\underline\mu(\theta)^{-1}$ and $\overline{K}(\theta)$ are second-order random variables. An assumption of this type is convenient when $\lambda$ is an unbounded (\eg~Gaussian or lognormal) random variable at $\vec{x}\in D$. See \cite[Lemma 1.2]{Babuska.Nobile.ea:07} and \cite{Charrier:12} for a discussion in the context of elliptic PDEs with random data.
\end{remark}


\section{Probabilistic framework}\label{sec:prob.setting}

In many applications, only limited information about the poroelastic coefficients in \eqref{eq:Biot1} and \eqref{eq:Biot2} is available. In the context of geomechanics, \eg, even when it is possible to carry out a large number of measurements, the actual knowledge of the soil properties typically suffers from inaccuracies: due to the presence of different layers, the physical properties can have strong variations that are difficult to estimate.  

\subsection{Uncertain poroelastic coefficients}\label{sec:poro_coeffs}

In this section we present the probabilistic setting and we introduce the parametrization of the random coefficients using canonical random variables. Some preliminary investigations on the possible choices to model the uncertain coefficients are subsequently presented.

\subsubsection{Parametrization with canonical variables}

We use a set of canonical random variables, collected into a random vector $\vec\xi=(\xi_1,\ldots,\xi_N):\Theta\to\Xi$, to describe the uncertainty of the poromechanical coefficients. To ensure the independence of the random coefficients, it suffices to use one canonical random variable per coefficient, and rely on changes of measure to define the mappings between the $\xi_i$ and their respective coefficients. 
Specifically, we deal with uniformly \textit{iid} canonical variables, namely
$\xi_i\sim \mathcal{U}([-1,1])$ and $\xi_i\perp\xi_j$ if $i\neq j$. 
Thus, the $\xi_i$'s have zero mean, are independent with product form  probability law $\mathcal{P}_{\vec\xi}$ and joint density function $\rho(\vec\xi)$. 
We denote by $\mathcal{B}_{\vec\xi}$ the Borel $\sigma$-algebra on $\Xi$ and by $(\Xi,\mathcal{B}_{\vec\xi},\mathcal{P}_{\vec\xi})$ the image probability space. 
We also define the space of second-order random variables on $(\Xi,\mathcal{B}_{\vec\xi},\mathcal{P}_{\vec\xi})$ as the weighted Lebesgue space $L^2_{\rho}(\Xi)$.
The expectation operator on the image space is denoted using brackets and is related to the expectation on $(\Theta,\mathcal{B},\mathcal{P})$ through the identity
$$
  \langle h\rangle \coloneq \int_\Xi h(\vec{\xi})\rho(\vec{\xi}) \ud\vec\xi =
  \int_\Theta h(\vec\xi(\vec{\theta})) \ud\mathcal{P}(\vec{\theta}) = \mathbb{E}(h\circ\vec{\xi}).
$$
The variance operator of $h\in(\Xi,\mathcal{B}_{\vec\xi},\mathcal{P}_{\vec\xi})$ is then defined by
$$
  \Var(h)\coloneq \left\langle (h -\langle h\rangle)^2\right\rangle.
$$

\subsubsection{Probabilistic model}
In order to account for the uncertainty of the poroelastic material, a model for its properties is needed. 
For the sake of simplicity, we assume that the fluid bulk modulus and the porosity are 
deterministic: $\Kf=2.2\ {\rm GPa}$ and $\varphi=0.2$.
Our input uncertainty is then parametrized by a vector $\vec\xi$ of dimension $N=4$.
The set of random coefficients consists in the two Lam\'e's constants $\lambda(\vec\xi)$ and $\mu(\vec\xi)$, the Biot--Willis coefficient $\alpha(\vec\xi)$, and the hydraulic mobility $\kappa(\vec\xi)$. These coefficients will be considered all mutually independent, with no spatial variabilities. The case of spatial variability can be handled by introducing a Karhunen--Lo\`eve expansion of the corresponding stochastic fields~\cite{Loeve:77,Schwab.Todor:06} but involving a high number of input random variables.  
Even if the experiments presented in this work are mainly academic model problems, we try to consider realistic sets of parameters and have a particular regard for the underlying physical phenomena. The ranges of variation of the uncertain input data are inspired from \cite[Table 1]{Rice.Cleary:76}, \cite[Section 3]{Zimmerman:00}, and \cite[Table 4.1]{Coussy:04}, and correspond to an ideal water filled sand soil with a $20\%$ reference porosity. 
In particular, for the numerical investigations of Section \ref{sec:tests.validation}, we assume that $\mu,\lambda$, and $\kappa$ have a log-uniform distribution, while $\alpha$ is uniformly distributed. We set
 \begin{equation}\label{eq:tc_parameters}
   \begin{aligned}
     \mu(\vec{\xi}) &= 10^{(\xi_1+1)} \ {\rm kPa},
     \\
     \lambda(\vec{\xi}) &= 2\cdot10^{(\xi_2+1)} \ {\rm kPa},
     \\
     \alpha(\vec{\xi}) &= \frac{\alpha_{\rm max}+\alpha_{\rm min}}{2} + \xi_3 \frac{\alpha_{\rm max}-\alpha_{\rm min}}{2},
     \\
     \kappa(\vec{\xi}) &= 10^{(\xi_4-1)} \ {\rm m}^2 {\rm kPa}^{-1} {\rm s}^{-1}.
   \end{aligned}
 \end{equation}
 According to \eqref{eq:Biot_coeffs} and \eqref{eq:bnd.Zimmerman}, we take 
 $\alpha_{\rm min}=3\varphi(2+\varphi)^{-1}$ and $\alpha_{\rm max}=1$. This choice yields 
 \begin{equation}\label{eq:mean.cv}
 \begin{aligned}
 \langle\mu\rangle &\sim 21.5\ {\rm kPa},\quad\langle\lambda\rangle \sim 43\ {\rm kPa},\quad\langle\alpha\rangle\sim 0.64, \;
 \text{ and }\quad\langle\kappa\rangle \sim 0.22 \frac{{\rm m}^2}{{\rm kPa}\ {\rm s}},
 \\
 {\rm c_v}(\mu) &\sim 1.16,\quad\quad{\rm c_v}(\lambda) \sim 1.16,
 \quad\, {\rm c_v}(\alpha)\sim 0.33, \;\text{ and }\;{\rm c_v}(\kappa)\sim 1.16,
 \end{aligned}
 \end{equation}
where ${\rm c_v}(\cdot)\coloneq\frac{\Var(\cdot)^{\nicefrac12}}{\langle \cdot\rangle}$ denotes the coefficient of variation operator. In particular we observe that the previous average values have the same magnitude as the coefficients considered 
in~\cite[Table 2]{Gaspar.Lisbona.ea:07}. The case of other standard distributions can be handled similarly introducing the inverse of the cumulative distribution function of the considered coefficient.

According to the results presented in Section \ref{sec:storage_uncertain}, the constrained specific storage coefficient $c_0$ can be expressed in terms of the elastic moduli $\lambda,\mu$ and the coupling Biot coefficient $\alpha$ (see also \cite{Cosenza.ea:02,Green.Wang:90} for theoretical and empirical investigations on the storage coefficient) as well as the porosity and the fluid bulk modulus. Therefore, we let
\begin{equation}
  \label{eq:random_coeffs.c0}
  c_0(\xi)\coloneq c_0(\mu(\vec\xi),\lambda(\vec\xi),\alpha(\vec\xi),\varphi,\Kf).
\end{equation}
The advantage of this approach is threefold: 
\begin{inparaenum}[(i)]
 \item it produces a strategy of perturbation with independent model coefficients,
 \item it ensures that the set of poroelastic coefficients belong to the physical admissible set, and
 \item it allows to reduce the uncertainty dimension and then the size of the sampling (\ie~the number of sparse grid points).
\end{inparaenum}
We illustrate here the definition of $c_0(\vec\xi)$ in \eqref{eq:random_coeffs.c0}, reporting its probability distribution for random poromechanical coefficients as in \eqref{eq:tc_parameters}. 
For each realization of the random poromechanical coefficients, we solve \eqref{eq:2order_c0} in order to compute the corresponding value of $c_0$. 
\begin{figure}
  \centering
  \includegraphics[height=6.2cm]{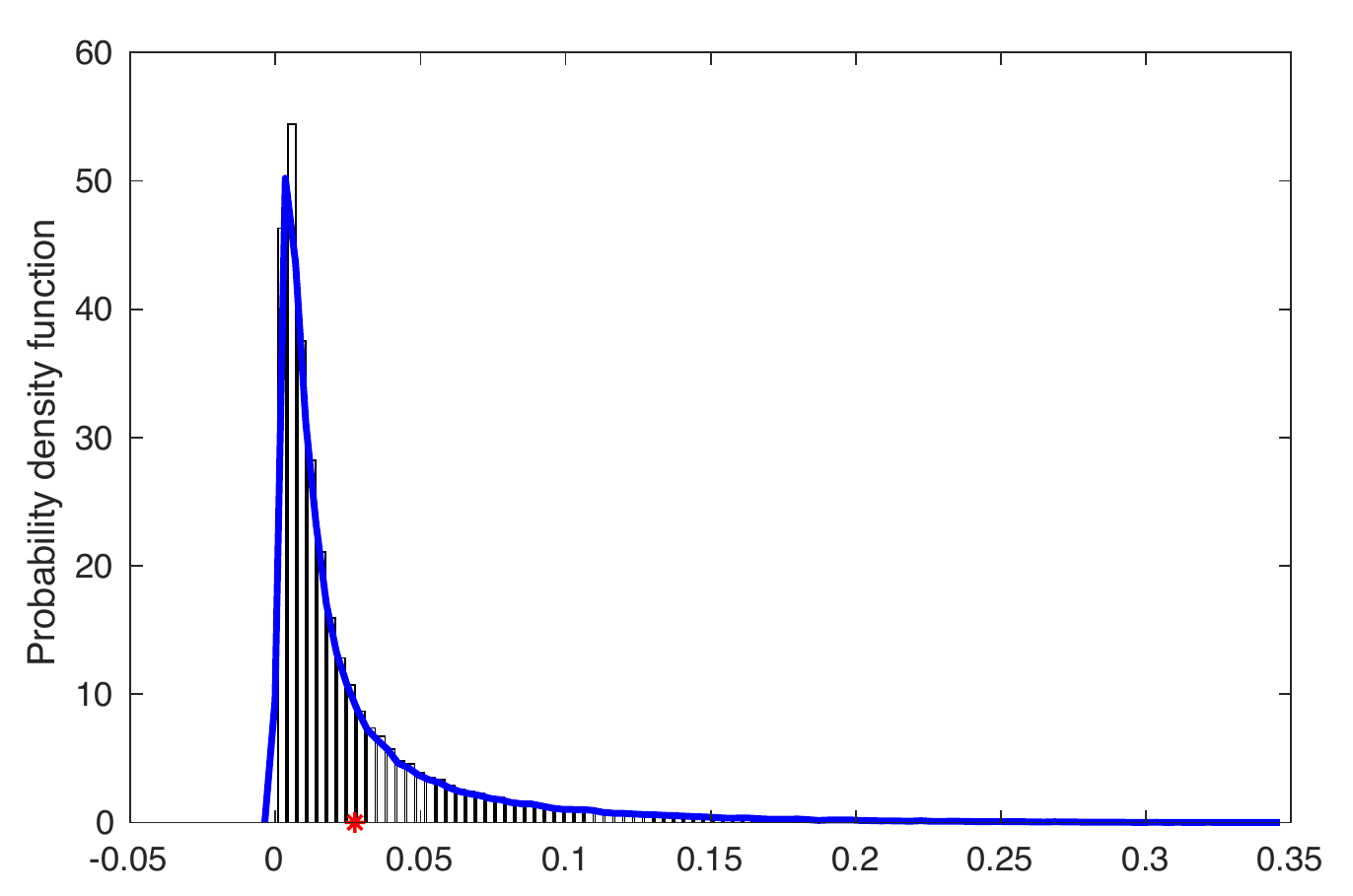}
  \caption{Distribution of $c_0$ obtained by solving \eqref{eq:2order_c0} with random input coefficients $\lambda, \mu$, and $\alpha$ distributed according to \eqref{eq:tc_parameters} (data are in ${\rm kPa}^{-1}$). Densities are estimated using $10^5$ realizations.\label{fig:model_comp}}
\end{figure} 
In Figure \ref{fig:model_comp} we plot the histogram and probability density function of $c_0$, estimated using $10^5$ realizations of the poromechanical coefficients.
We observe that our choice leads to a specific storage coefficient $c_0$ ranging in $(2\cdot10^{-4},3\cdot10^{-1})\ {\rm kPa}^{-1}$ and its probability density function is far from being uniform with a highest probable value close to zero.

Different choices are possible to describe the mechanical behavior of the material and, in practice, the selection of the particular coefficients defining the elastic properties of the medium is mostly one of convenience.
For instance, one might consider the Young modulus $E$ (\eg~in \cite{Khan.Powell.ea:18}) or the bulk modulus $K$ and the Poisson ratio $\nu$ (\eg~in \cite{Rice.Cleary:76,Cosenza.ea:02}), as the independent driving coefficients.
Here we choose instead Lam\'e's coefficients  $\mu$ and $\lambda$ as primary independent variables, because they explicitly appear in the governing equations and are also more conveniently defined through simple transformations of independent canonical variables. 
The corresponding couple $(K,\nu)$ can be deduced from Lam\'e's coefficients $(\mu,\lambda)$ using the relations 
$$
  K(\vec\xi)=\frac2d\mu(\vec\xi) + \lambda(\vec\xi)
  \quad\text{ and }\quad
  \nu(\vec\xi)=\frac{\lambda(\vec\xi)}{2(\mu(\vec\xi) + \lambda(\vec\xi))}. 
$$
Figure~\ref{fig:mulam_nonuniform} shows the marginals of $K$ and $\nu$ for independent primary mechanical variables $(\mu,\lambda)$ distributed according to \eqref{eq:tc_parameters}. 
\begin{figure}
  \centering
  \begin{minipage}{0.48\textwidth}\centering
    \includegraphics[height=6cm, width=7.5cm]{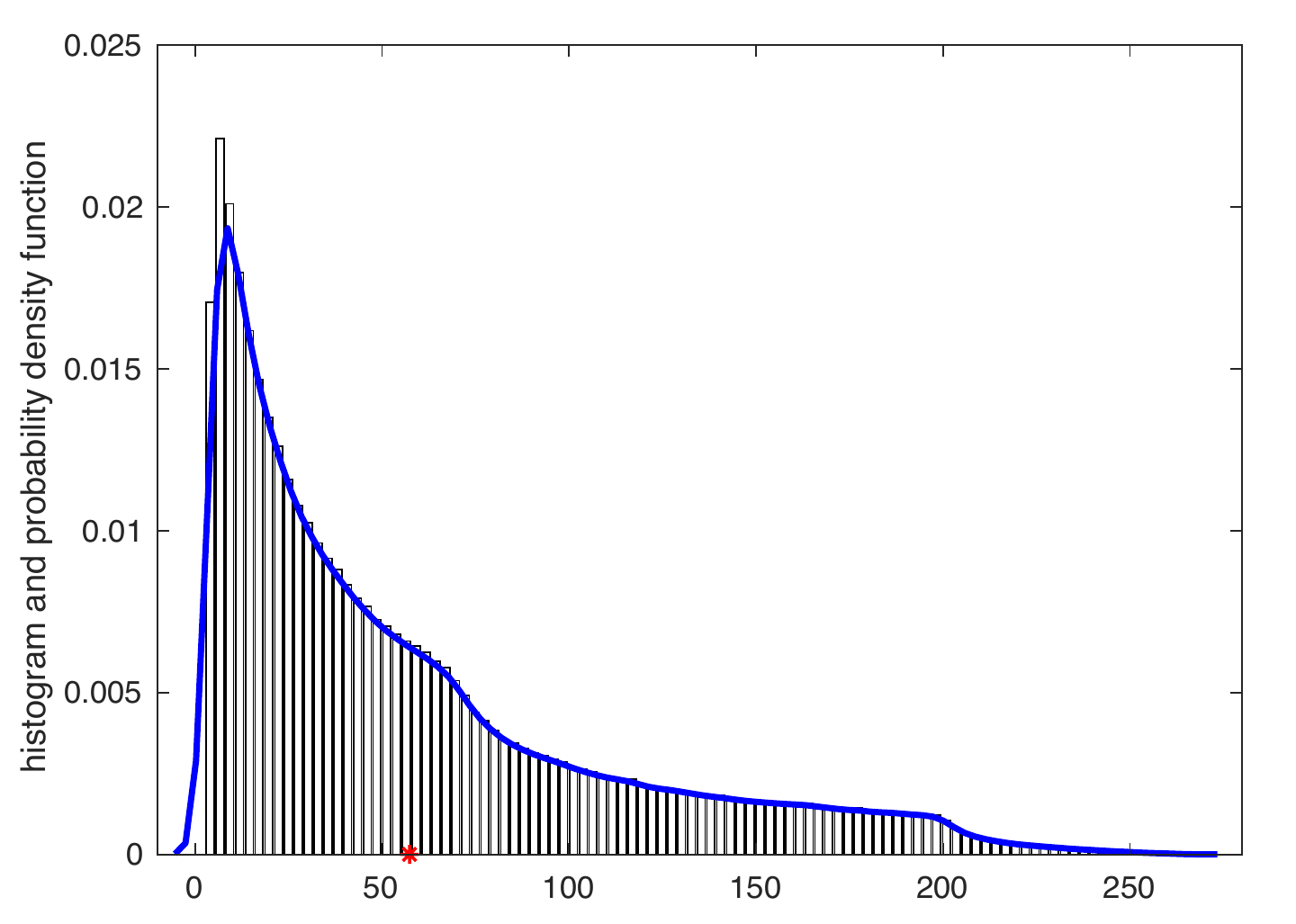}
  \end{minipage} 
  \hspace{2mm}
  \begin{minipage}{0.48\textwidth}\centering
    \includegraphics[height=6cm, width=7.5cm]{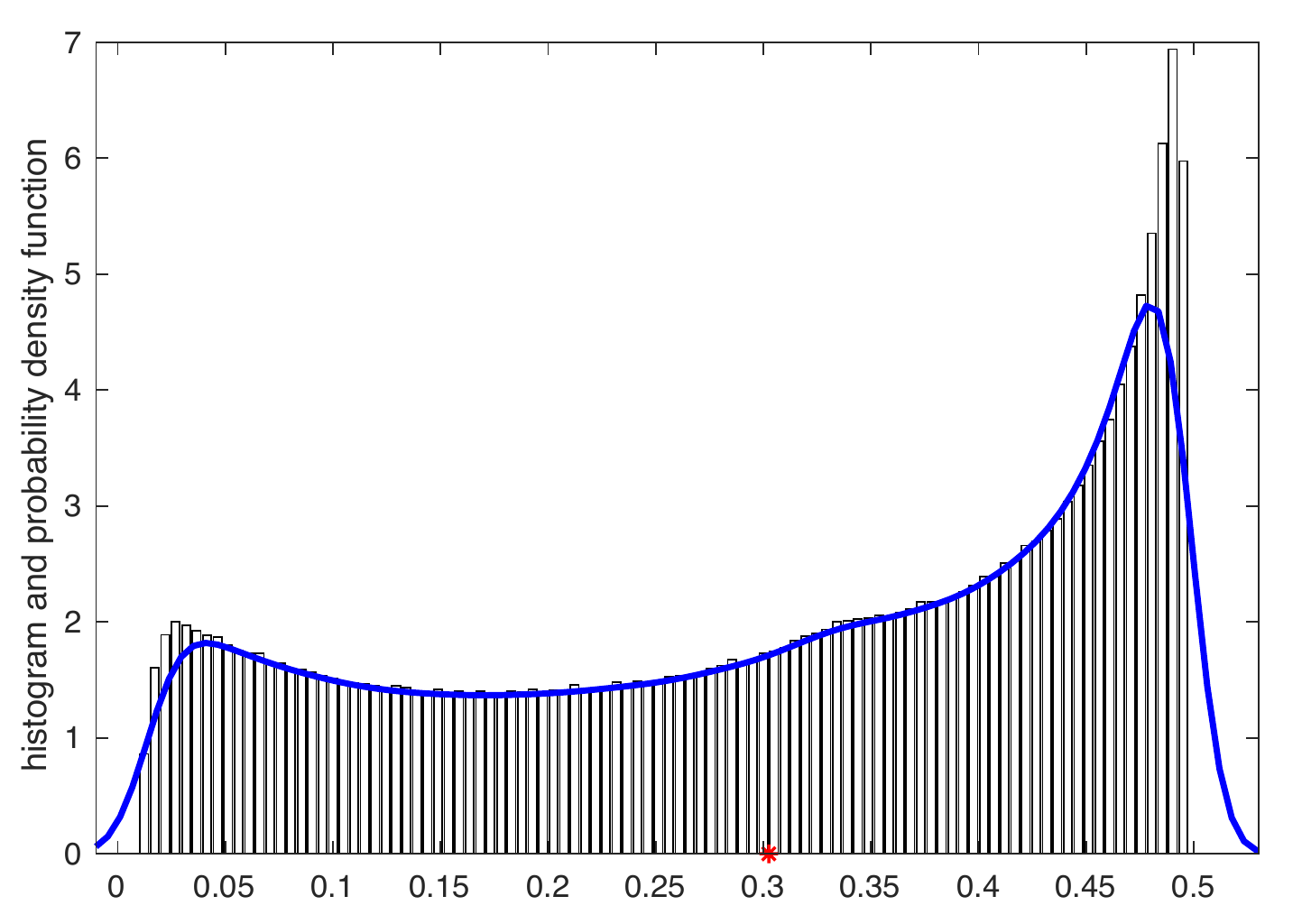}
  \end{minipage} 
  \caption{Marginal distributions of the bulk modulus $K$ (left) and the Poisson ratio $\nu$ (right) corresponding to $10^6$ realizations for $\mu$ and $\lambda$ distributed as in \eqref{eq:tc_parameters}.\label{fig:mulam_nonuniform}}
\end{figure}
Starting from independent $K$ and $\nu$ results in a complex dependent distribution of the coefficients $\lambda$ and $\mu$ appearing in the model equations. This is due to the non-linear relation $(K,\nu) \mapsto (\lambda,\mu)$ that can be appreciated from Figure~\ref{fig:mulam_scatter} for the case where $K$ is log-uniformly distributed and $\nu$ is uniformly distributed. 
\begin{figure}
  \centering
    \includegraphics[height=6.1cm]{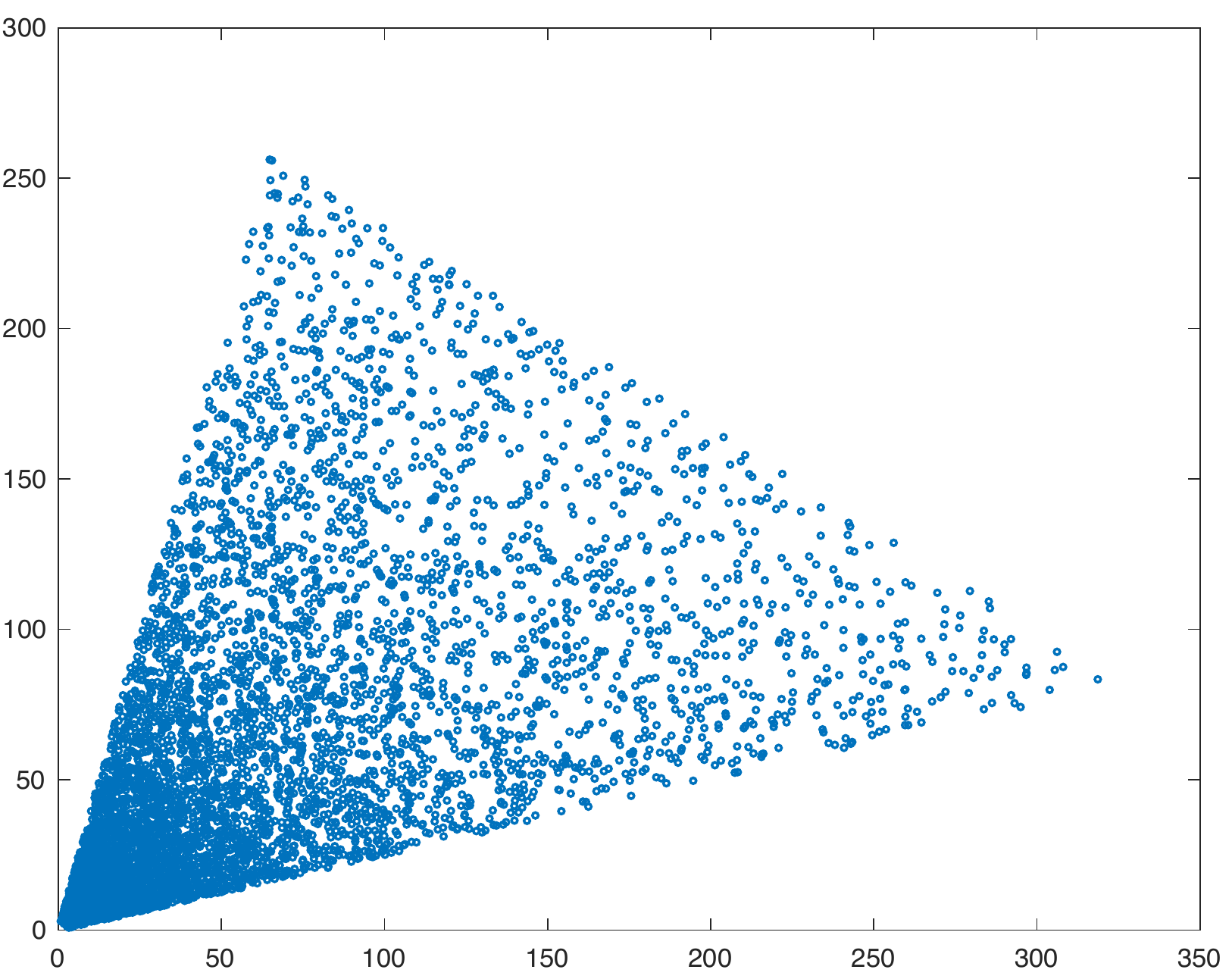}
  \caption{Scatter plot of $10^4$ realizations $(\mu,\lambda)$ for $K\sim\mathcal{LU}([3, 300]) \ {\rm GPa}$ and $\nu\sim\mathcal{U}([0.1, 0.4])$. 
\label{fig:mulam_scatter}}
\end{figure}


\subsection{Stochastic discretization}\label{sec:discrete_setting}
Now that the uncertain coefficients and their parametrization have been detailed, we introduce in this section the Polynomial Chaos expansions of the random model solution and outline the Pseudo-Spectral Projection algorithm used for its determination. The PSP method is one of the so-called non-intrusive techniques, which rely on an ensemble of deterministic model simulations to estimate the expansion coefficients of the solution. 

\subsubsection{Polynomial Chaos expansion}
PC expansions were initially proposed by Wiener \cite{Wiener:38} for the approximation of second order random quantities from standard \textit{iid} Gaussian variables. It was later proposed to use these polynomial expansions for UQ~\cite{Ghanem:91}. Concerning the generalization to non-Gaussian measure and non-polynomial expansions, see \cite{Xiu:02} and references within \cite{LeMaitre:10}. 

Let us denote by $\{ \phi_{\vec{k}}(\vec{\xi}) : \vec{k}\in \Natural^N\}$ an Hilbertian basis of $L^2_\rho(\Xi)$, where $\phi_{\vec{k}}$ is a multivariate polynomial in $\vec{\xi}$ and $\vec{k}=(k_1,\ldots, k_N)$ is a multi-index indicating the polynomial degree in the $\xi_i$'s. The total degree of $\phi_{\vec{k}}$ is $|\vec{k}|_1\coloneq\sum_{i=1}^N k_i$. The basis functions are commonly chosen to be orthogonal with respect to the inner product in $L^2_\rho(\Xi)$ characterized by the probability density function $\rho:\Xi\to\Real^+$,
\begin{equation}\label{eq:PC_orthogonality}
  \langle \phi_{\vec{k}},\phi_{\vec{l}}\rangle = 
  \int_{\Xi} \phi_{\vec{k}}(\vec{\xi}) \phi_{\vec{l}}(\vec{\xi}) \rho(\vec{\xi})\ud\vec{\xi}
  =\delta_{\vec{k},\vec{l}}. 
\end{equation}
In our case, where $\Xi$ is the hypercube $[-1,1]^N$ with uniform density, the $\phi_{\vec k}$ are in fact the multivariate Legendre polynomials. If $X(\vec{\xi})\in L^2_\rho(\Xi)$ is a second-order random variable, then it admits the PC expansion~\cite{Wiener:38,Cameron:47},
\begin{equation}\label{eq:Modedef1}
    X(\vec{\xi})= \sum_{\vec{k}\in\Natural^N} X_{\vec{k}} \phi_{\vec{k}}(\vec{\xi}), 
\end{equation}
where the deterministic coefficients of the series $\{X_{\vec{k}} : \vec{k}\in \Natural^N\}$ are called the \textit{spectral modes}. The PC approximation $X_{\mathcal{K}}(\vec{\xi})$ of $X(\vec{\xi})$ is obtained by truncating the expansion above to a finite series,
\begin{equation}\label{eq:PCdef}
X_{\mathcal{K}}(\vec{\xi}) \coloneq \sum_{\vec{k}\in\mathcal{K}} X_{\vec{k}}\phi_{\vec{k}}(\vec{\xi}),
\end{equation}
where $\mathcal{K}\subset\Natural^N$ is a finite set of multi-indices.
Observing that the degree zero polynomial is $\phi_{\vec{0}} =1$, the expectation and variance of $X(\vec \xi)$ simply express as
\begin{align}
  \label{eq:PCexpectation}
  \langle X_{\mathcal{K}} \rangle = X_{\vec0}, \quad
  \Var(X_{\mathcal{K}}) =\sum_{\vec{k}\in\mathcal{K}\setminus\vec0} X_{\vec{k}}^2. 
\end{align}
Different truncation strategies, \ie~the selection of the set $\mathcal{K}$, can be considered, with different resulting error $e_{\mathcal{K}}(\vec{\xi}):=X(\vec{\xi})-X_{\mathcal{K}}(\vec{\xi})$.
In particular, if the approximation is defined as the $L^2$-orthogonal projection of $X$, the spectral coefficients are defined by 
\begin{equation}\label{eq:Modedef}
  X_{\vec{k}}\coloneq \langle X,\phi_{\vec{k}}\rangle,
\end{equation}
and it is clear that the $L^2$-error norm does not increase when $\mathcal K$ is enlarged. Indeed, we have in this case
$$
\norm[L^{2}_\rho(\Xi)]{X-X_{\mathcal{K}}}^2 = \sum_{\vec k \in \Natural^N \setminus \mathcal K} |X_{\vec k}|^2.
$$
A classical truncation strategy is based on the degree of the expansion; for instance, one can define $\mathcal K$ such that the approximation basis contains all multivariate polynomials with total ($|\vec k|_1\le p$) or partial 
($|\vec k|_\infty\coloneq\max_{1\le i \le N} k_i \le p$) degree up to a prescribed value $p\ge 0$. The convergence rate of the PC expansion of $X$ with respect to the degree $p$ depends on the regularity (in $\vec \xi$) of $X$. In particular, an exponential convergence is expected for analytical variables; more details about the convergence conditions can be found in \cite{Cameron:47,Ernst:12}. 

\subsubsection{Sparse Pseudo-Spectral Projection}

Several methods have been proposed to compute the spectral modes of the PC expansion~\cite{LeMaitre:10,Iskandarani:16}. Among these methods, we distinguish between the intrusive and non-intrusive ones. The former involve a reformulation of the problem to derive a set of governing equations for the solution modes, while the latter rely on the availability of a deterministic solver only. 
As a first attempt to apply PC expansions to propagate uncertainty in the Biot model, we opted for a non-intrusive projection method which, in our experience, presents the best trade-off between numerical complexity and precision, compared to alternatives such as regression and compressed sensing. 

The Non-Intrusive Spectral Projection (NISP) method is based on approximating the correlation in the right-hand-side of \eqref{eq:Modedef} by a deterministic numerical quadrature formula, such as
\begin{equation}\label{eq:Modes}
  X_{\vec{k}} = \langle X,\phi_{\vec{k}}\rangle 
  \simeq \sum_{q=1}^{N_q} {w^{(q)} X(\vec{\xi}^{(q)}) \phi_{\vec{k}}(\vec{\xi}^{(q)}) },
\end{equation}
where the $N_q$ quadrature nodes $\vec{\xi}^{(q)}$ and weights $w^{(q)}$ are classically constructed by tensorization of one-dimensional quadrature rules (\textit{e.g.} Gauss, F\'ej\`er, Clenchaw--Curtis, \dots). 
The complexity of the NISP method is governed by the number $N_q$ of evaluations of the model, while the accuracy depends on the basis which can be considered given the quadrature rule. 
Clearly, these two aspects are intertwined, as larger expansion bases require higher order quadrature rules with larger computational efforts. 
In practice, the basis is defined as the largest one such that the discrete projection is exact for any function belonging to the linear-span of the basis or, in other words, that the non-intrusive projection incurs no internal aliasing. Equivalently, the multi-index set $\mathcal{K}$ must satisfy 
$$
  \forall \vec{k}, \vec{l} \in \mathcal{K}, \: \sum_{q=1}^{N_q} {w^{(q)} \phi_{\vec{k}}(\vec{\xi}^{(q)})
  \phi_{\vec{l}}(\vec{\xi}^{(q)}) } = \delta_{\vec{k}\vec{l}}.
$$

Exploiting the product structure of the stochastic space, sparse tensorizations of one-dimensional quadrature rules~\cite{Gerstner:98,Gerstner:03} have been proposed to mitigate computational complexity, especially for a larger number $N$ of canonical random variables. The sparse constructions generally rely on the Smolyak's formula~\cite{Smolyak:63} and lead to the so-called sparse-grid projection methods. One drawback of using sparse quadrature rules, as opposed to fully tensorized quadrature rules, is that some of the weights $w^{(q)}$ are negative even for one-dimensional formulas with positive weights only. 
Therefore, the sparse quadratures are not constituting discrete inner products, with some restrictions on the possible aliasing-free bases as a result. 

These limitations were recognized in~\cite{Conrad:13,Constantine:12}, where the authors proposed the so-called Pseudo-Spectral Projection (PSP) methods. 
The key-idea of PSP is to apply the Smolyak's formula directly on the projection operator, rather than on the integration operator. The projection is seen as a sequence of nested subspace-projections, where each subspace-projection is computed using a fully-tenzorized quadrature rule. Consequently, the determination of each spectral mode uses an adapted fully-tensorized quadrature rule with improved properties compared to the unique one in \eqref{eq:Modes}. 
Specifically, for the same sparse grid, with $N_q$ nodes $\vec{k}^{(q)}$, the PSP method can compute without internal aliasing the projection coefficients $X_{\vec{k}}$ of a larger set $\mathcal K$ of basis functions $\phi_{\vec{k}}$, compared to the NISP approach. 

In this work, we then use a sparse PSP method, constructed on nested one-dimensional Clenshaw--Curtis quadrature rules. As the canonical random variables account for uncertainties in different poromechanical coefficients whose respective influences on the model solution are unknown, we decided to rely on an isotropic construction. In the PSP context, it amounts to consider isotropic partial tensorization, with a unique parameter controlling the complexity and accuracy of the approximation: the maximum level $l\in \Natural$ of isotropic the sparse grid. As the level $l$ of the PSP method increases, both the number of sparse grid nodes $N_q(l)$ and the multi-index set $\mathcal{K}(l)$ increase.  


\section{Test cases}\label{sec:tests}

We now assess the performance of the PSP method detailed in the previous section to treat the stochastic part of the Biot problem~\eqref{eq:SysBiot} due to parametric uncertainties.
Recall that the PSP method requires to solve $N_q$ deterministic Biot problems (discretized by the method of~\cite{Boffi.Botti.Di-Pietro:16}) for different values of uncertain input parameters. Note that the PSP method can be applied with alternative solvers, possibly based on a non-monolithic formulation (\textit{i.e.}, splitting methods). 
Our numerical experiments feature a simplified model with spatially homogeneous random coefficients. 
More sophisticated uncertain models, requiring (possibly strongly correlated) spatially varying input coefficients with non-standard distributions will be the object of future works. We begin this section with two validation cases; the first one has a fluid forcing term while a mechanical boundary condition drives the second one. An injection-production test case at large spatial scale is then presented to assess the methodology on a more realistic problem and perform a sensitivity analysis to illustrate the interest of the method.

\subsection{Validation test cases}\label{sec:tests.validation}
To start, we implement two experiments to study the convergence of the PSP method on the Biot problem. These two test cases are complementary in the sense that they consider a fluid or mechanical solicitation of the porous medium. More precisely, the first case is an injection test designed to investigate the effect of a fluid source on the deformation of the domain, while the second case is meant to evaluate the mechanical effect of a traction term on the fluid pressure. 
These two cases are defined on the unit square domain $D=[0,1]^2$ and the loading term $\vf$ in~\eqref{eq:SysBiot1} as well as the initial condition $\phi_0$ in~\eqref{eq:SysBiot:initial} are both equal to zero. The probability distributions of the random poroelastic coefficients are defined in Section~\ref{sec:poro_coeffs}. After describing the test cases, we numerically investigate the convergence of the stochastic error with respect to the level of the sparse grid. We also analyze the variance and covariance fields of the PC solutions.  

\subsubsection{Point injection}
The first experiment is inspired by the Barry and Mercer test case~\cite[Section 6.2]{Boffi.Botti.Di-Pietro:16} (see also ~\cite{Barry.Mercer:99,Phillips.Wheeler:07*1}) in which the displacement and pressure fields are driven by a stationary fluid source
$$
g(\vec{x}) =10\ \delta(\vec{x}-\vec{x}_0),
$$
where $\vec{x}_0=(0.25,0.25)$ denotes the source location and $\delta$ the Dirac delta function. 
Homogeneous mechanical and fluid boundary conditions are enforced on the whole boundary of the domain,
$$
\begin{aligned}
  \vu\cdot\vec{\tau}&=0,&\text{on }\partial D,\\
  \GRAD\vu\normal\cdot\normal&=0,&\text{on }\partial D,\\
  p &= 0,&\text{on }\partial D,
\end{aligned}
$$
where $\vec{\tau}$ denotes the tangent unit vector on $\partial D$ such that $\vec{\tau}\perp\normal$ with arbitrary orientation.
The following results have been obtained using the method of~\cite{Boffi.Botti.Di-Pietro:16} with polynomial degree $k=3$ on a Cartesian mesh containing $\pgfmathprintnumber{1024}$ elements. Using the static condensation procedure~\cite[Section 5]{Boffi.Botti.Di-Pietro:16}, the discrete system has $\pgfmathprintnumber{27648}$ unknowns. We are interested in the stationary solution, namely the pressure and displacement fields manifesting after the initial transient phase, and report the solution at the final time $\tF=1 {\rm s}$ obtained after $10$ time steps of the implicit Euler scheme.  
\begin{figure}
  \centering
  \begin{minipage}{0.47\textwidth}\centering
    \includegraphics[height=6cm]{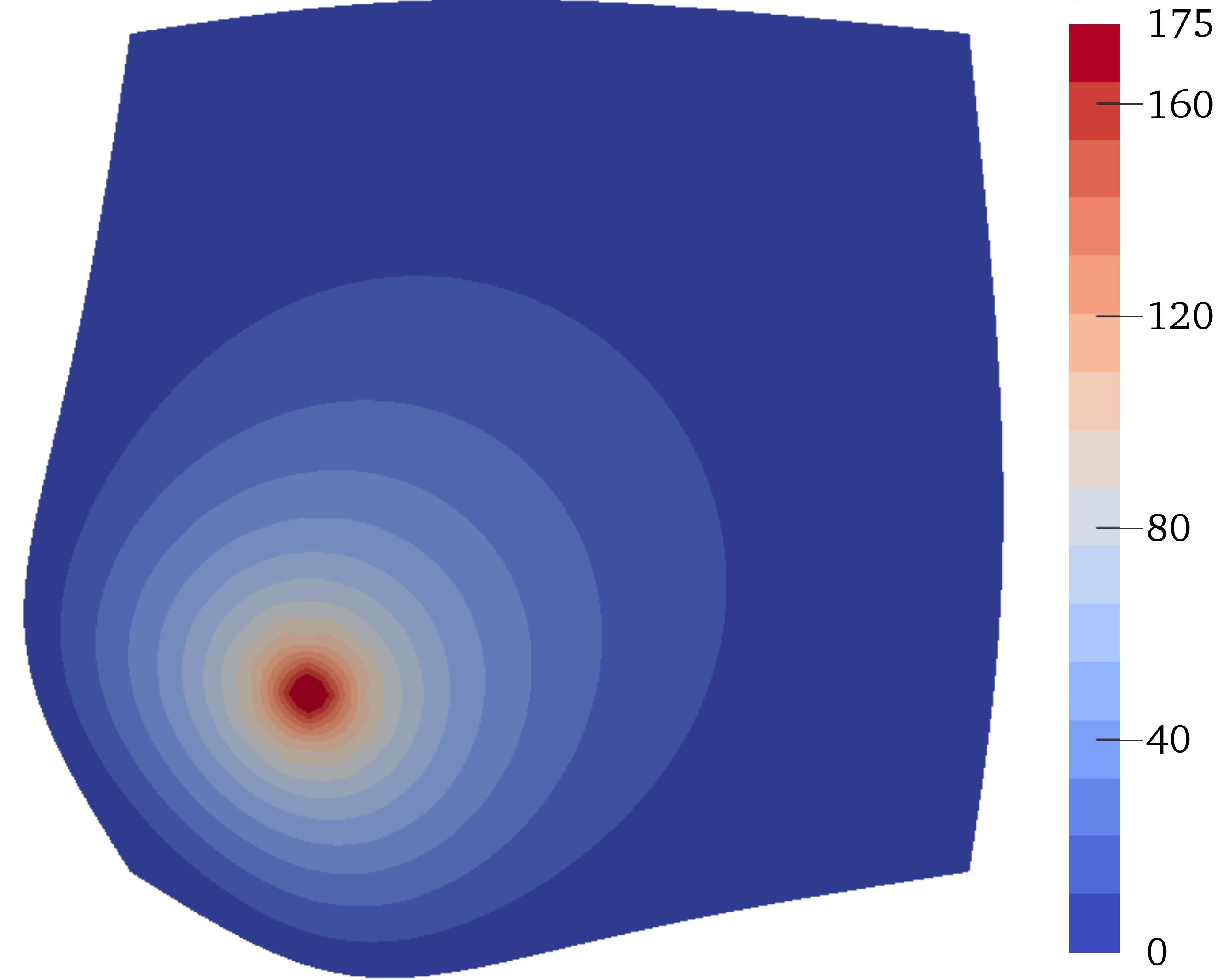}
  \end{minipage} 
  \hspace{2mm}
  \begin{minipage}{0.47\textwidth}\centering
    \includegraphics[height=6cm]{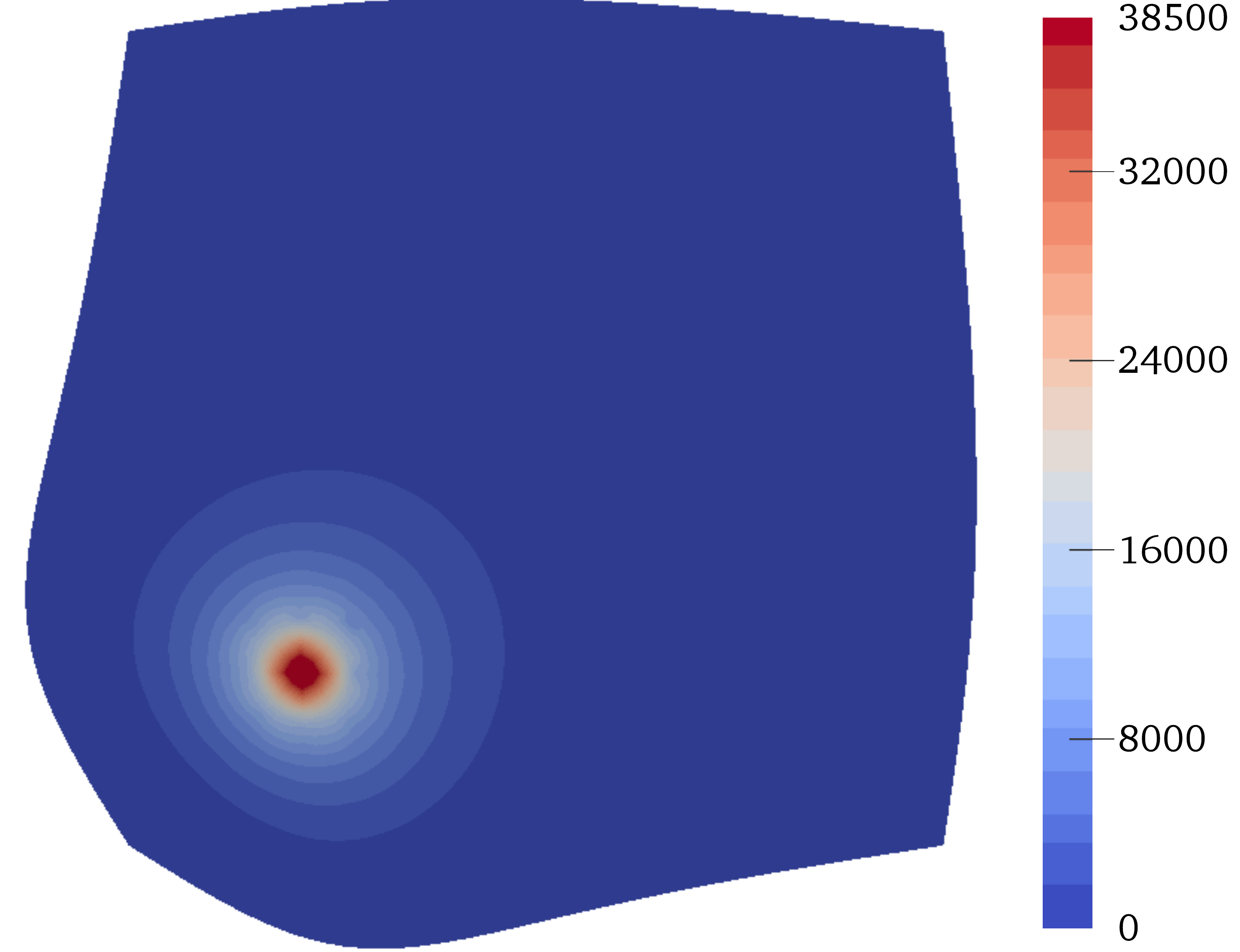}
  \end{minipage}
  \caption{Mean (left) and variance (right) of the pressure field plotted on the deformed domain at $t=\tF=1 {\rm s}$ (data are in ${\rm kPa}$). 
   Solution for the PSP method with level $l=5$ and $N_q(l)=2561$.
  \label{fig:TC1_MeanVar}}
\end{figure}

Figure~\ref{fig:TC1_MeanVar} represents the two first statistical moments~\eqref{eq:PCexpectation} of the pressure field approximated using a level $l=5$ sparse grid. 
We plot the solution on the deformed configuration obtained by applying the (scaled) computed mean displacement field to the reference unit (undeformed)  square domain $D$.
As expected, we observe a dilatation of the domain caused by the injection. 
The computed solutions are symmetric with respect to the diagonal according to the source location and the homogeneous boundary conditions and to the fact that the computational mesh also has the same symmetry.
The mean pressure field exhibits maximum values at the source, while only small perturbations are observed near the boundary. 
The variance of the pressure field has the same behavior as the mean, with higher values close to the injection point.
We notice that the magnitude of the standard deviation, \ie~the square root of the variance, is comparable to the mean field range.  
We draw attention to the fact that the variance (and covariance) fields must be carefully computed since the spatial dependence involves the products of broken polynomial functions on the mesh and its skeleton.  
 
\subsubsection{Poroelastic footing}
As a second example, we examine the 2D footing problem proposed in~\cite{Murad.Loula:94,Gaspar.Lisbona.ea:07}, where the domain is assumed to be free to drain and fixed along the base and vertical edges. A uniform load is applied on a central portion of the upper boundary $\Gamma_{N}=\{\vec{x}=(x_1,x_2)\in\partial D : x_2=1\}$ simulating a footing step compressing the medium.
The boundary conditions are defined as
$$
\begin{aligned}
  \ms\normal &= (0, -5\ {\rm kPa})\,\quad \text{on } 
  \Gamma_{N,1}\coloneq\{\vec{x}\in\partial D \st 0.3\le x_1\le 0.7, \ x_2=1\},
  \\
  \ms\normal &=\vec0, \qquad\qquad\quad \text{on } \Gamma_N\setminus\Gamma_{N,1},
  \\
  \vu&=\vec0,\qquad\qquad\quad \text{on }\partial D\setminus\Gamma_N,
  \\
  p &= 0,\qquad\qquad\quad\text{on }\partial D.
\end{aligned}
$$
As $\vf=\vec0$ and $\phi_0=0$, the displacement and pressure fields are determined only by the stress boundary condition. 
We set $k=2$ in the method of~\cite{Boffi.Botti.Di-Pietro:16} and discretize the domain using a triangular mesh with $\pgfmathprintnumber{3584}$ elements. 
We focus here on the pressure profile at early times and, for each sparse grid node, the linear system of dimension $\pgfmathprintnumber{54720}$ is solved twice to reach $\tF=0.2 {\rm s}$. 
Because of the free drainage boundary condition, the fluid is promptly squeezed out of the soil to reach an equilibrium configuration with no pressure and where the displacement balances the Neumann condition.  
We mention, in passing, that the spurious pressure oscillations observed in~\cite{Murad.Loula:94} do not occur with our discretization in agreement with results in~\cite[Section 6.2]{Boffi.Botti.Di-Pietro:16}. 

Figure~\ref{fig:TC2_MeanVar} shows the mean and variance of the pressure field over the deformed domain. The deformation of the domain corresponds again to the mean displacement field. The domain is particularly warped on the load segment and the maximum pressure is found near the center of the domain. As in the previous test case, we observe that the magnitude of the standard deviation is comparable to the mean field range.

\begin{figure}
  \centering
  \begin{minipage}{0.47\textwidth}\centering
    \includegraphics[height=6.2cm]{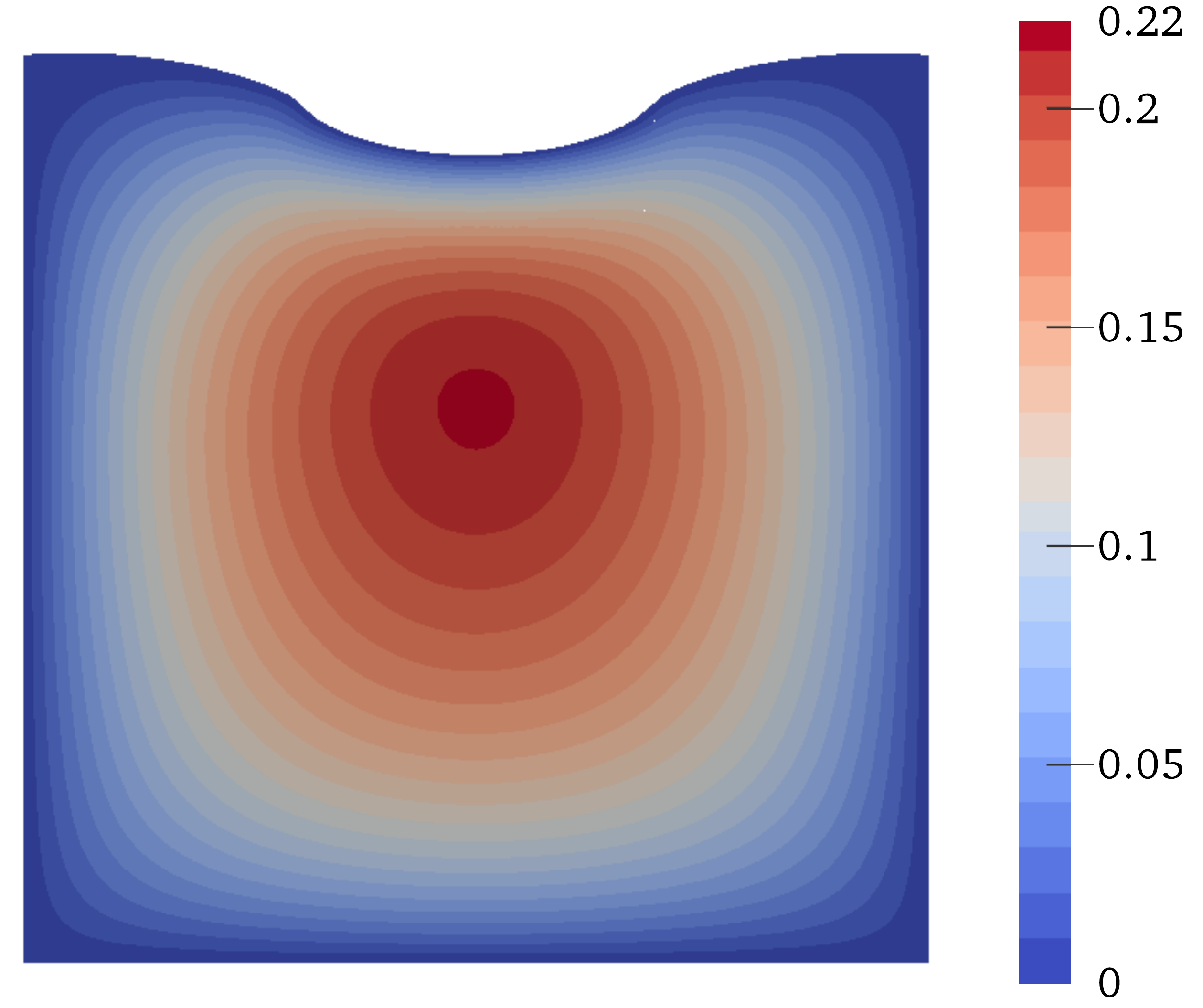}
  \end{minipage} 
  \hspace{2mm}
  \begin{minipage}{0.47\textwidth}\centering
    \includegraphics[height=6.2cm]{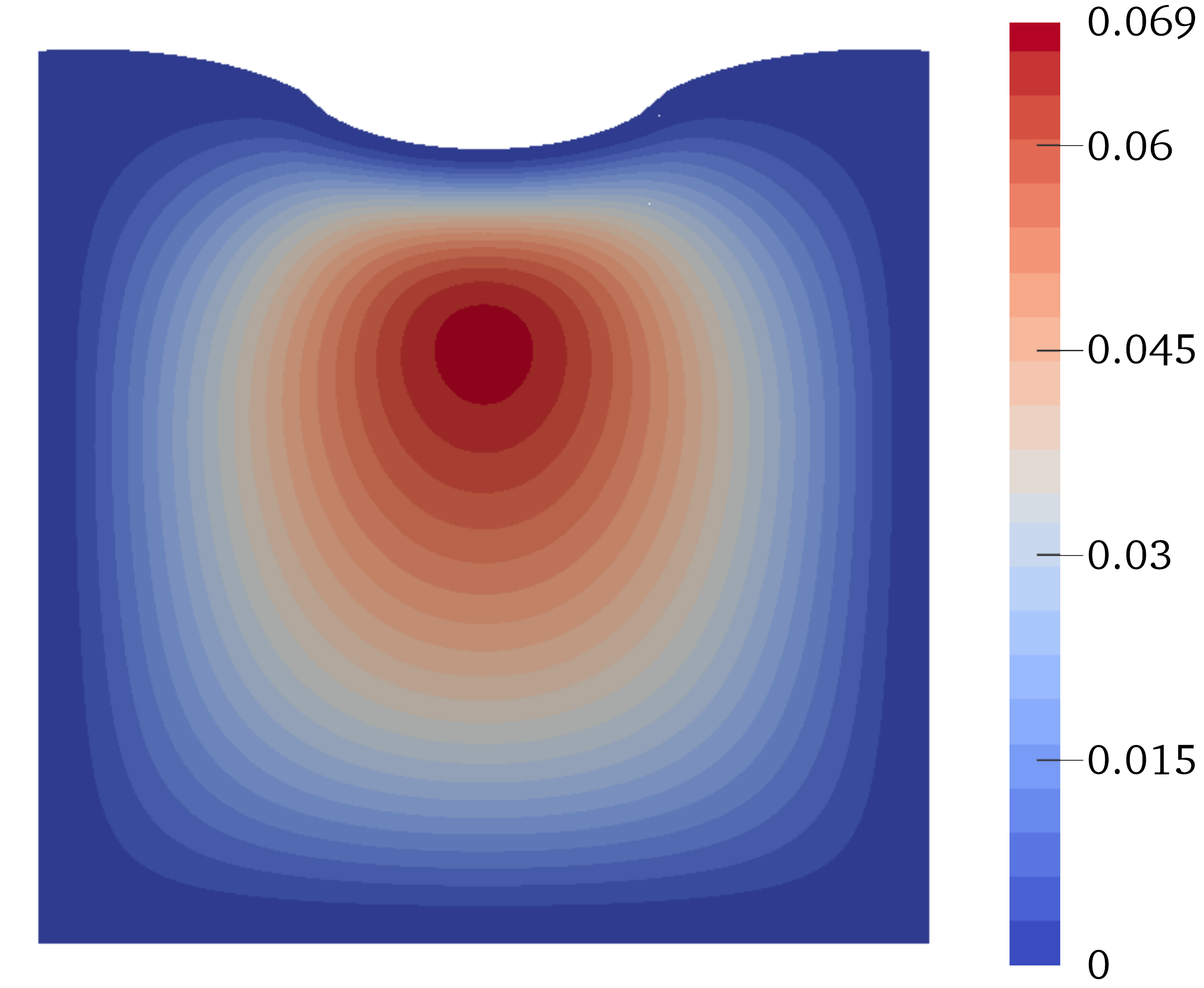}
  \end{minipage}
  \caption{Mean (left) and variance (right) of the pressure field plotted on the deformed domain at $t=0.2 {\rm s}$ (pressures are in ${\rm kPa}$). Solution computed using the PSP method with level $l=5$ and $N_q(l)=2561$.
  \label{fig:TC2_MeanVar}}
\end{figure}

\subsubsection{Convergence analysis}
We evaluate the accuracy of the PC expansions of the stochastic solution by using a validation set independent from the sparse grid nodes. The Mean-Squared Error (MSE) is computed from $N^*=500$ Latin Hypercube Samples (LHS), denoted $\vec\xi^{(i)}$,
$$
{\rm MSE}(X-X_{\mathcal{K}}) (\vec{x}) 
\coloneq\frac{1}{N^*} \sum_{i=1}^{N^*} \left(X(\vec{x},\vec\xi^{(i)})-{X}_{\mathcal{K}}(\vec{x},\vec\xi^{(i)})\right)^2,
$$
where $X(\vec{x},\cdot)$ denotes the exact (pressure or displacement) field and $X_{\mathcal{K}}(\vec{x},\cdot)$  
its PC approximation. 
We plot on Figures~\ref{fig:TC1_ConvMSE} and~\ref{fig:TC2_ConvMSE} the MSE fields of the PC solutions. Owing to the symmetry of each configuration, we only represent the error of the first component of the displacement field for the injection test and the second displacement component for the footing test. The results are plotted for the sparse grids of level $l=1,3,5$, and we observe a rapid decrease of the error for the two test cases and both the displacement and pressure fields. We especially notice the strong decay of the maximum error value with the increasing sparse grid level.
To complete the convergence analysis, we also plot inn Figure~\ref{fig:TC1e2_Conv} the $L^2(D)$-norm of the mean-squared error fields, $\norm[{\Lvec{D}}]{{\rm MSE}(\vu-{\vu}_{\mathcal{K}})}$ and $\norm[L^2(D)]{{\rm MSE}(p-{p}_{\mathcal{K}})}$ up to a sparse grid level equal to 5. 
The linear trends observed in the semi-log scale of the successive PC approximations suggest that the method achieves roughly an exponential convergence rate with respect to the level of the sparse grid, as predicted by the theory for a sufficiently regular quantity of interest.

\begin{figure}
  \centering
  \includegraphics[height=5.5cm]{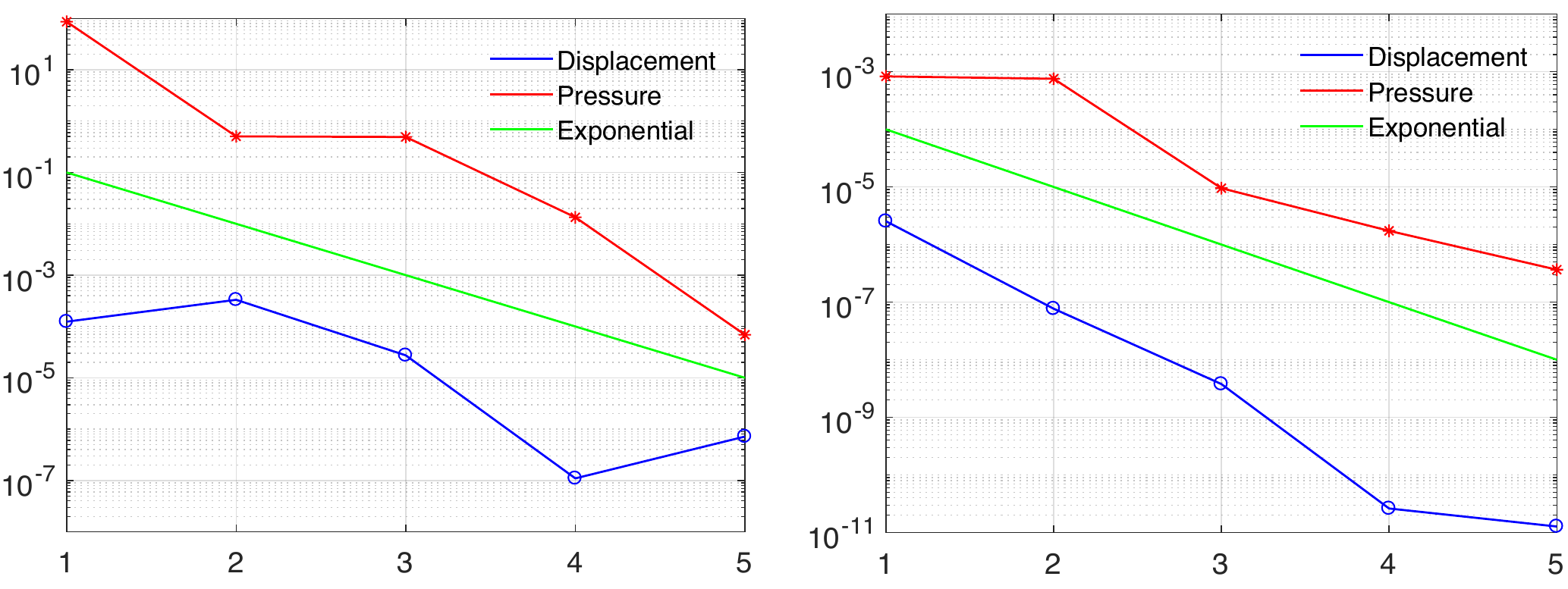}
  \caption{Errors $\norm[{\Lvec{D}}]{{\rm MSE}(\vu-{\vu}_{\mathcal{K}})}$ and $\norm[L^2(D)]{{\rm MSE}(p-{p}_{\mathcal{K}})}$ vs. level $l$ of the Sparse Grid. Point injection test (left) and footing test (right) with model coefficients $\mu,\lambda,\alpha,\kappa$ distributed according to~\eqref{eq:tc_parameters}.\label{fig:TC1e2_Conv}}
\end{figure}
\begin{figure}
  \centering
  \includegraphics[height=20cm]{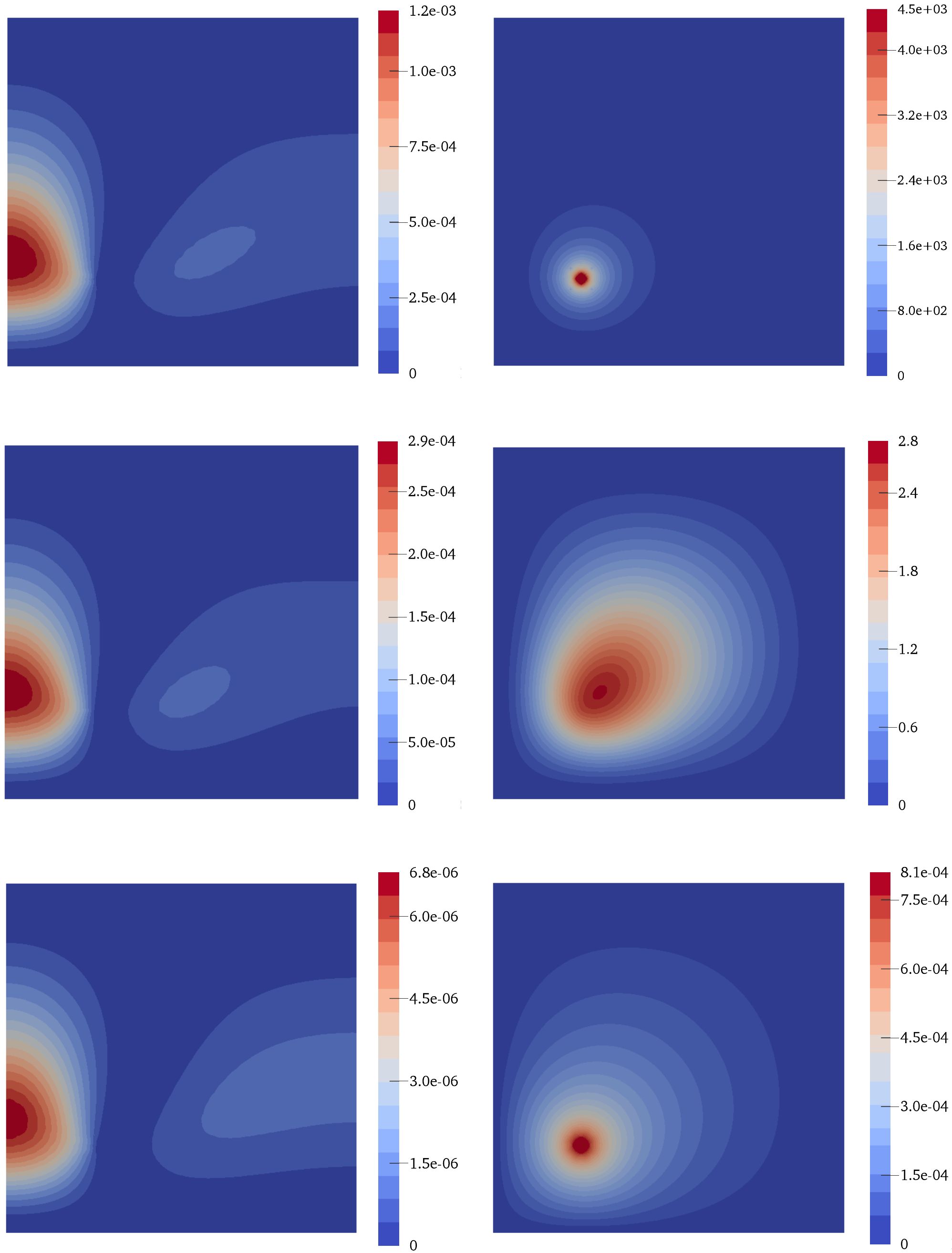}
  \caption{Displacement ${\rm MSE}(u_1-{u}_{\mathcal{K},1})$ field (left) 
    and pressure ${\rm MSE}(p-{p}_{\mathcal{K}})$ field (right) of the injection test case. Results are computed using the PSP method with $l=1$ (top), $l=3$ (middle), and $l=5$ (bottom).
    Notice the change of scales when increasing the sparse grid level.
    \label{fig:TC1_ConvMSE}}
\end{figure}
\begin{figure}
  \centering
  \includegraphics[height=20cm]{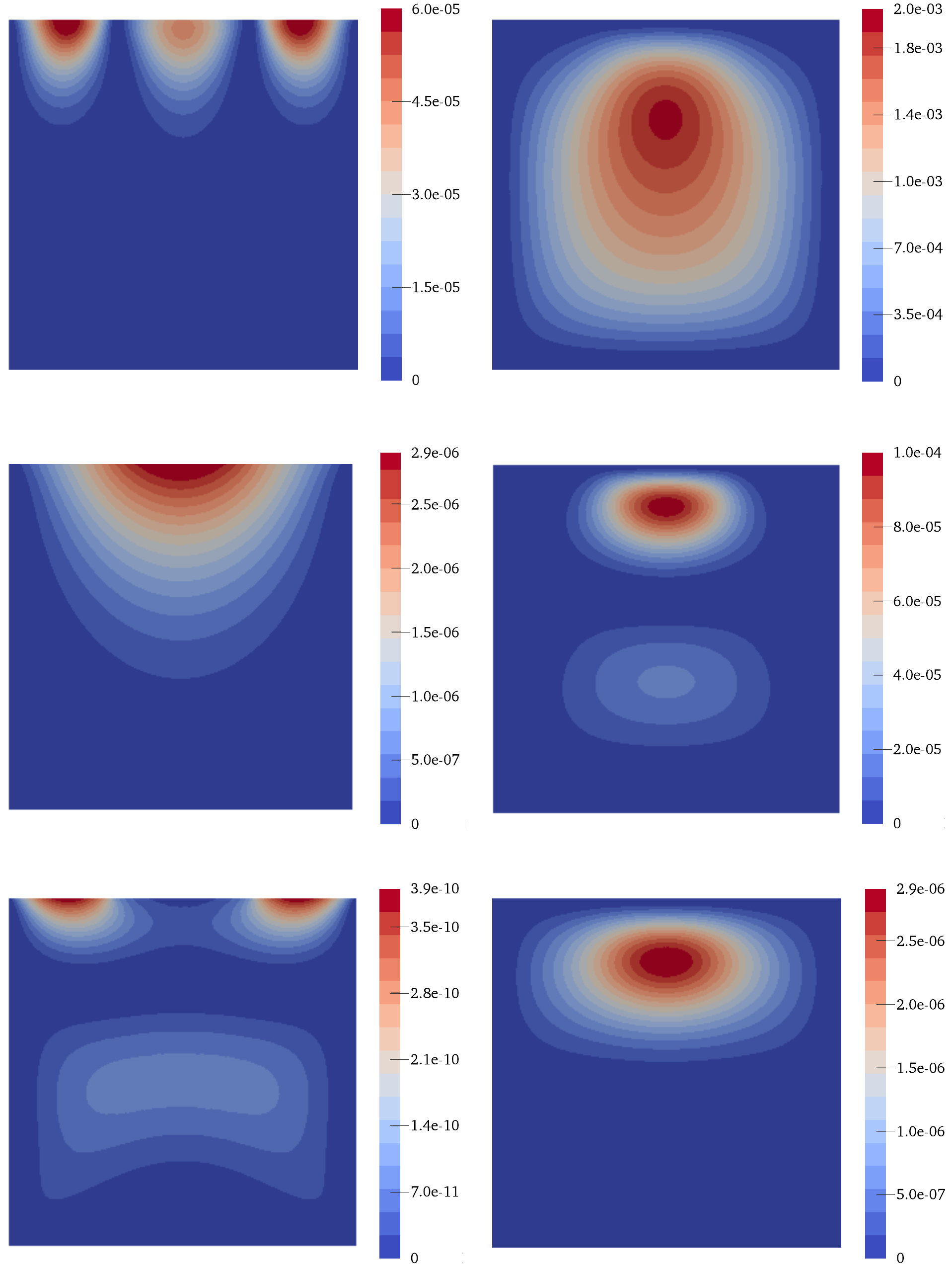}
  \caption{Displacement ${\rm MSE}(u_2-{u}_{\mathcal{K},2})$ field (left) 
  and pressure ${\rm MSE}(p-{p}_{\mathcal{K}})$ field (right) of the footing test case.
  Results computed using the PSP method with $l=1$ (top), $l=3$ (middle), and $l=5$ (bottom).
  Notice the change of scales when increasing the sparse grid level.
  \label{fig:TC2_ConvMSE}}
\end{figure}

\subsubsection{Covariance fields}
We finally analyze the correlation between the displacement and pressure fields. The covariance between two random variables $h_1,h_2\in(\Xi,\mathcal{B}_{\vec\xi},\mathcal{P}_{\vec\xi})$ is defined by $$\Cov(h_1,h_2)\coloneq \left\langle (h_1 -\langle h_1\rangle),(h_2 -\langle h_2\rangle)\right\rangle.$$
In the case of two random variables represeneted on by PC expansions in the same basis, $X_{\mathcal{K}}$ and $Y_{\mathcal{K}}$, the covariance reads
$$
\Cov(X_{\mathcal{K}},Y_{\mathcal{K}})=\sum_{\vec{k}\in\mathcal{K}\setminus\vec0} X_{\vec{k}}Y_{\vec{k}}.
$$

On Figure~\ref{fig:TC1_Cov} we plot the covariance fields of the injection test case. As expected, we observe a symmetry with the domain diagonal of the covariance field between the two displacement components (left plot). The sign of the covariance depends on the joint variation of the variables. 
Indeed the covariance field of ${u}_{\mathcal{K},1}$ and ${u}_{\mathcal{K},2}$ is in accordance with the fact that both $u_1$ and $u_2$ are negative on the bottom-left area, both $u_1$ and $u_2$ are positive on the top-right area, while either $u_1$ or $u_2$ is positive when the other component is negative on the top-left and bottom-right areas of the domain. 
High values of $\Cov({u}_{\mathcal{K},1},p_{\mathcal{K}})$ and $\Cov({u}_{\mathcal{K},2},p_{\mathcal{K}})$ (middle and right panel) are found close to the injection point, where the pressure and mechanical effects of the fluid source are maximum. We also remark that, owing to the symmetric configuration of the injection test, the covariance field $\Cov({u}_{\mathcal{K},1},p_{\mathcal{K}})$ corresponds roughly to the $90^\circ$ clockwise rotation of $\Cov({u}_{\mathcal{K},2},p_{\mathcal{K}})$. 

Figure~\ref{fig:TC2_Cov} represents the covariance fields for the footing test case. 
The covariance field of ${u}_{\mathcal{K},2}$ and $p_{\mathcal{K}}$ (right panel) has the same pattern as the pressure variance field plotted on Figure~\ref{fig:TC2_MeanVar} with the highest value below the loading segment. The covariance is here negative since the pressure increases during the vertical compression. On the middle panel, the covariance of the horizontal displacement ${u}_{\mathcal{K},1}$ and the pressure $p_{\mathcal{K}}$ is lower by one order of magnitude consistently with the vertical loading of this test case. The sign and the location of maximum and minimum values of $\Cov({u}_{\mathcal{K},1},{u}_{\mathcal{K},2})$ and $\Cov({u}_{\mathcal{K},1},p_{\mathcal{K}})$ (left and middle panel) are consistent with the deformed configuration plotted on Figure~\ref{fig:TC2_MeanVar}.
 \begin{figure}
   \centering
  \begin{minipage}{0.32\textwidth}\centering
    \includegraphics[height=4.5cm]{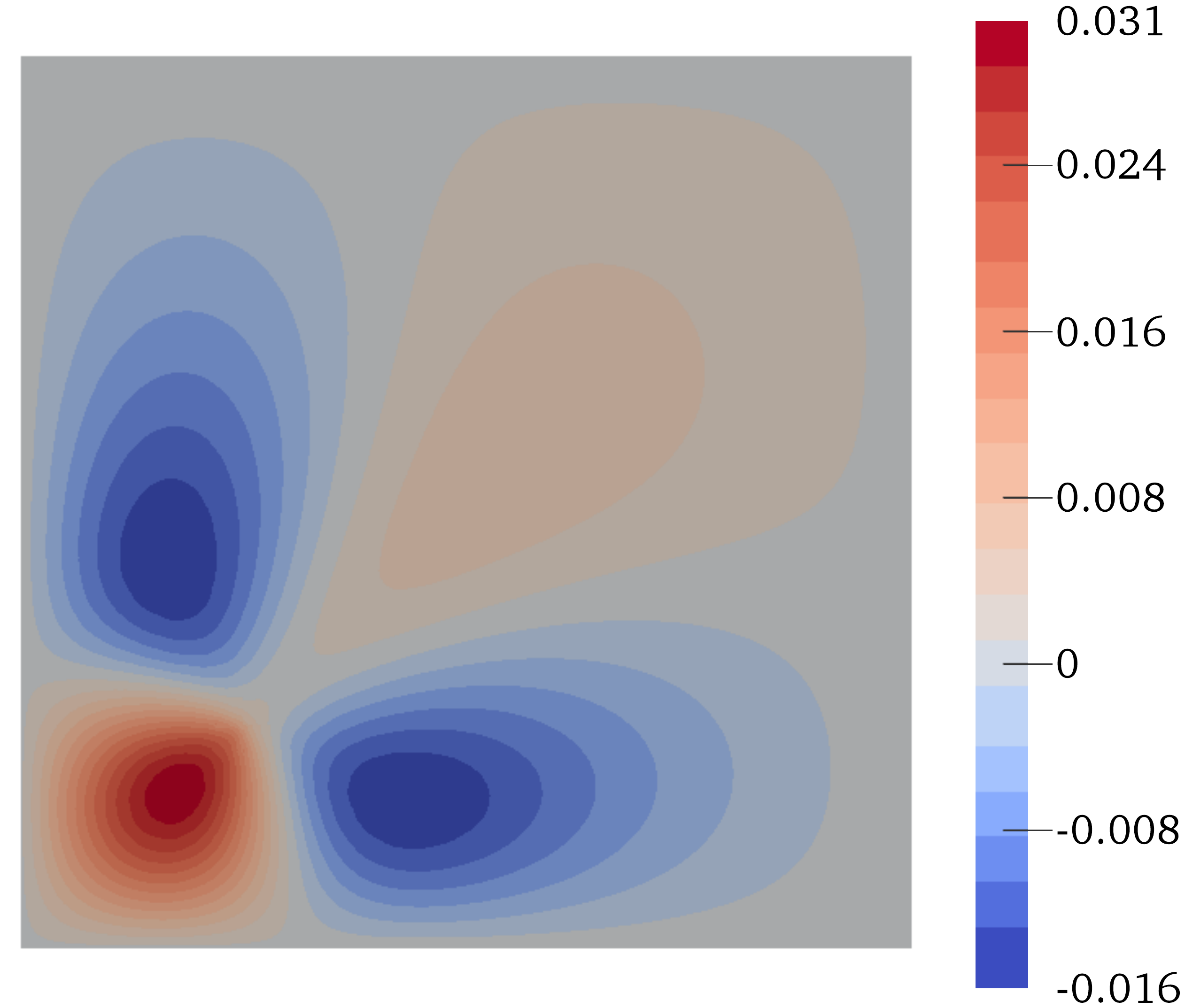}
  \end{minipage}  
  \hspace{1mm}
  \begin{minipage}{0.32\textwidth}\centering
    \includegraphics[height=4.5cm]{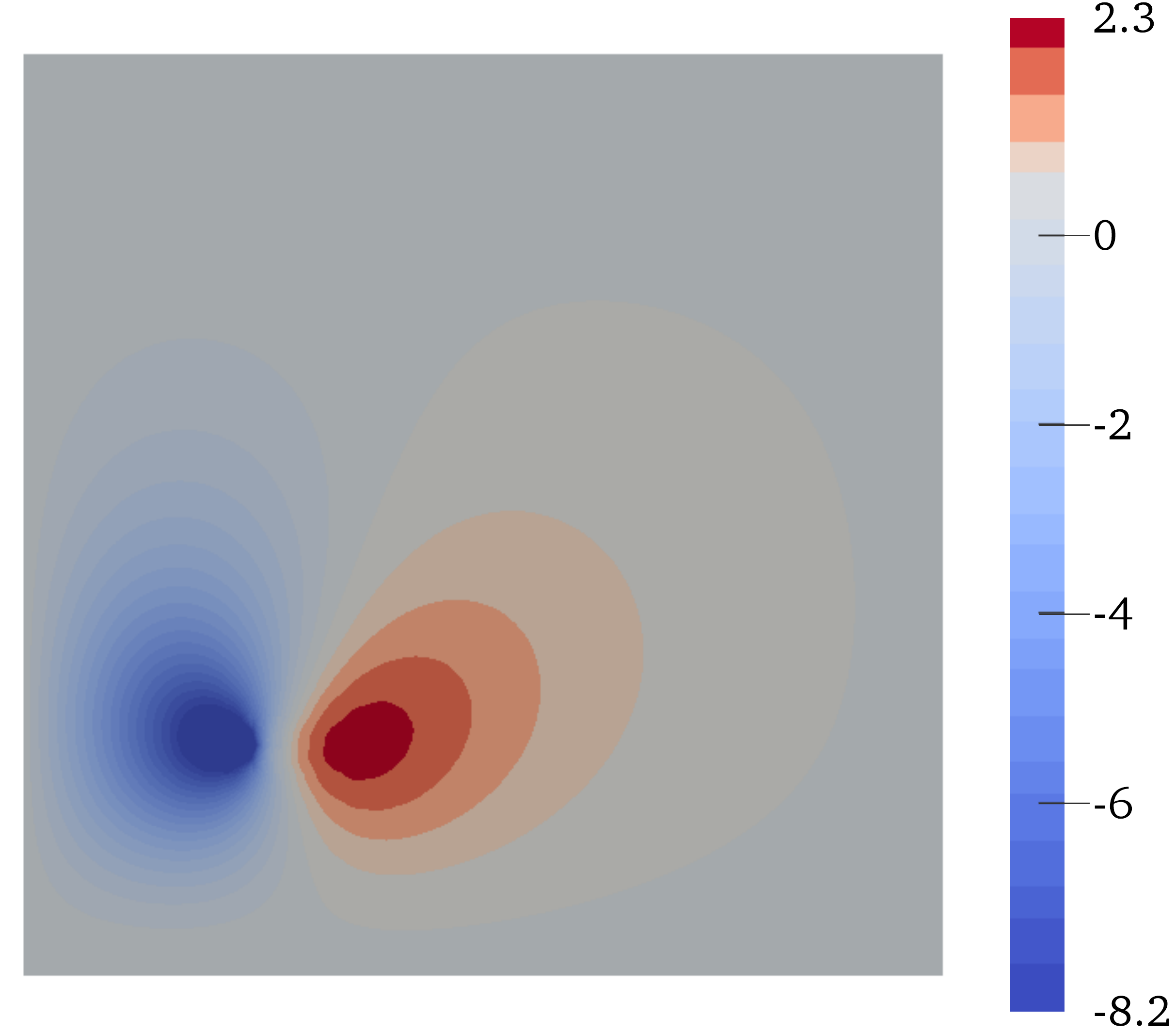}
  \end{minipage}
  \begin{minipage}{0.32\textwidth}\centering
    \includegraphics[height=4.5cm]{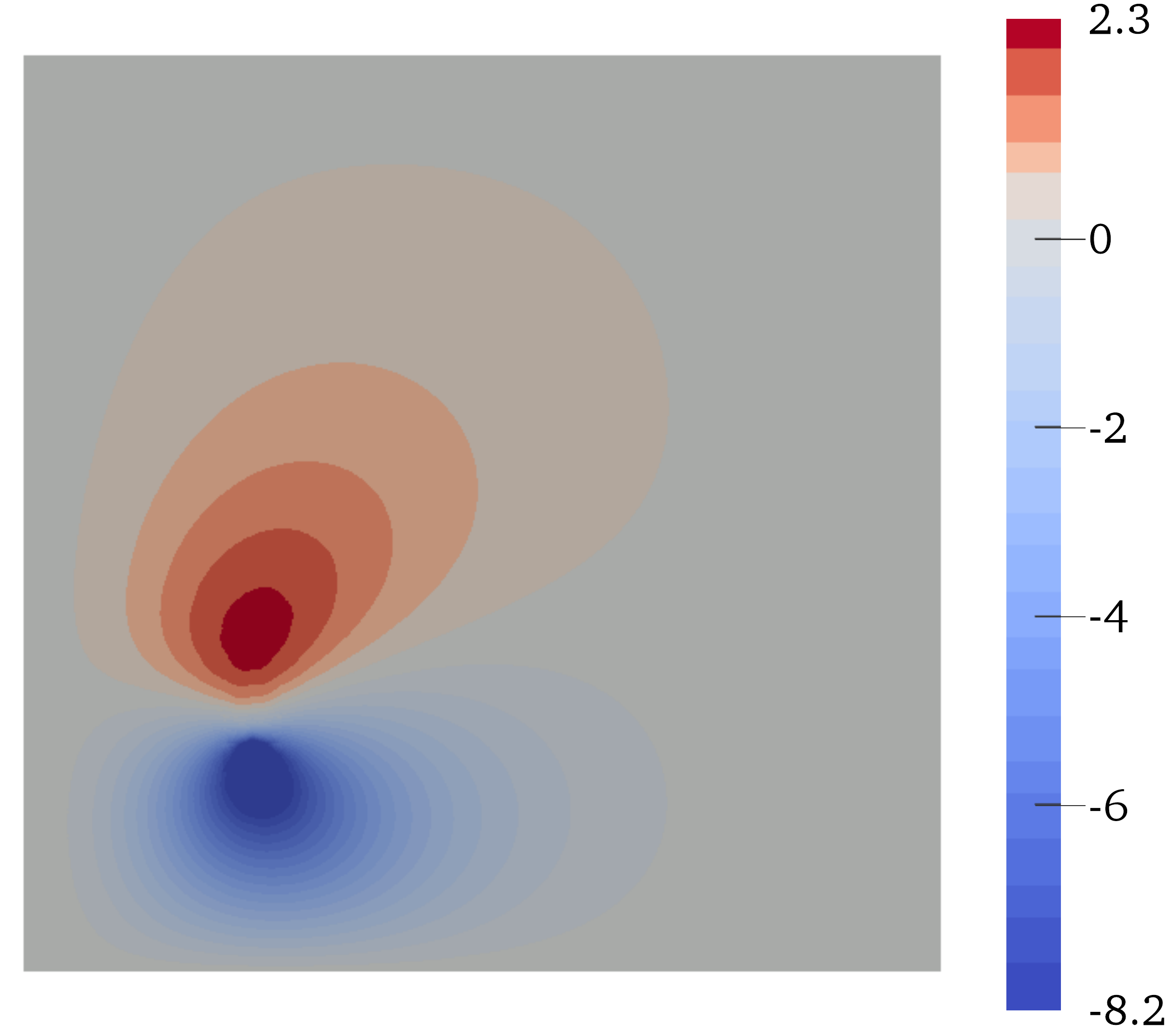}
  \end{minipage}
   \caption{$\Cov({u}_{\mathcal{K},1},{u}_{\mathcal{K},2})\ [{\rm m}^2]$ (left), $\Cov({u}_{\mathcal{K},1},p_{\mathcal{K}})\ [{\rm m}\cdot{\rm kPa}]$ (middle), and $\Cov({u}_{\mathcal{K},2},p_{\mathcal{K}})\ [{\rm m}\cdot{\rm kPa}]$ (right) fields of the injection test case computed at $\tF=1{\rm s}$ using the PSP method with level $l=5$ and $N_q(l)=2561$.}\label{fig:TC1_Cov}
 \end{figure}
 \begin{figure}
   \centering
  \begin{minipage}{0.32\textwidth}\centering
    \includegraphics[height=4.5cm]{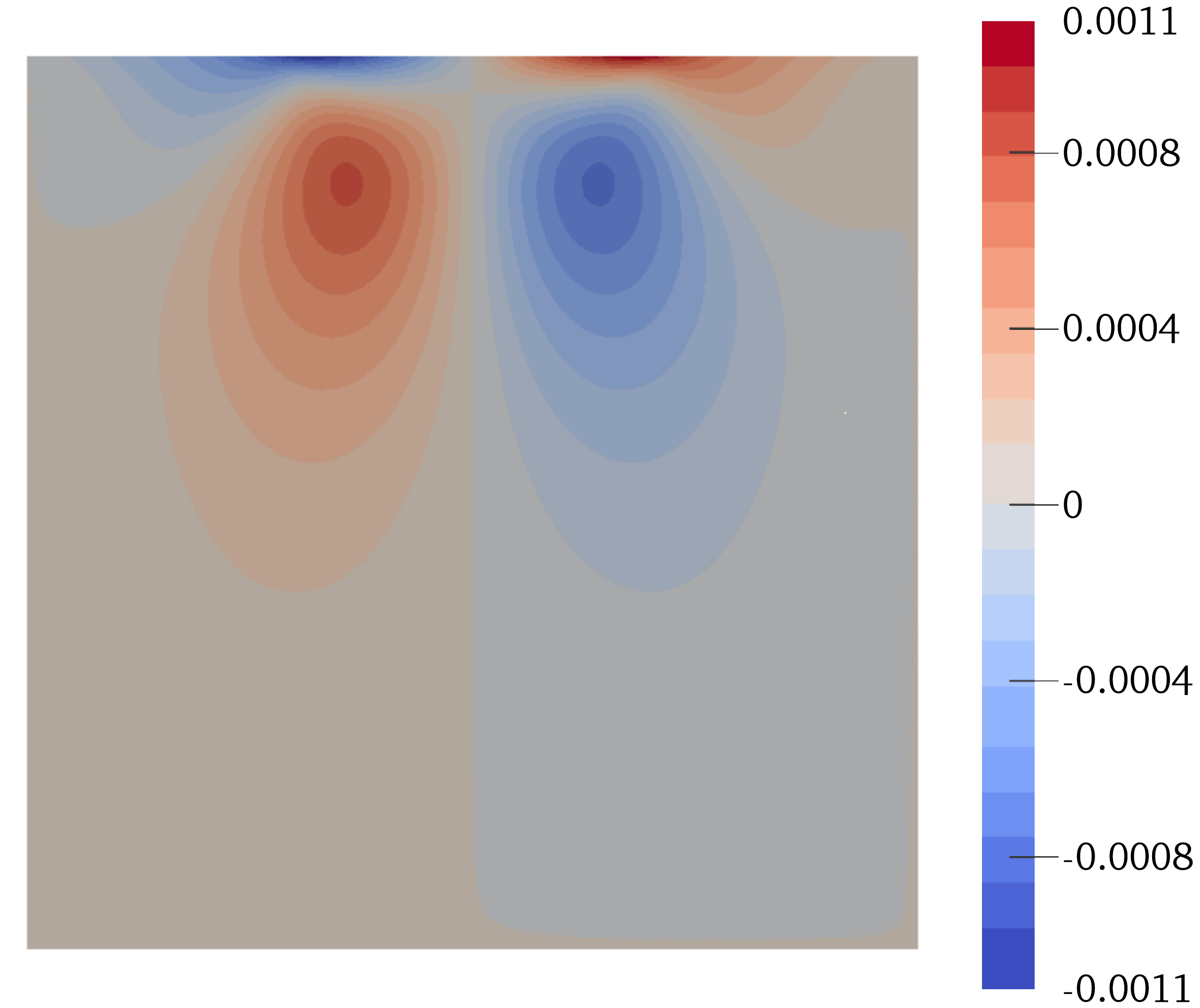}
  \end{minipage} 
  \hspace{1mm}
  \begin{minipage}{0.32\textwidth}\centering
    \includegraphics[height=4.5cm]{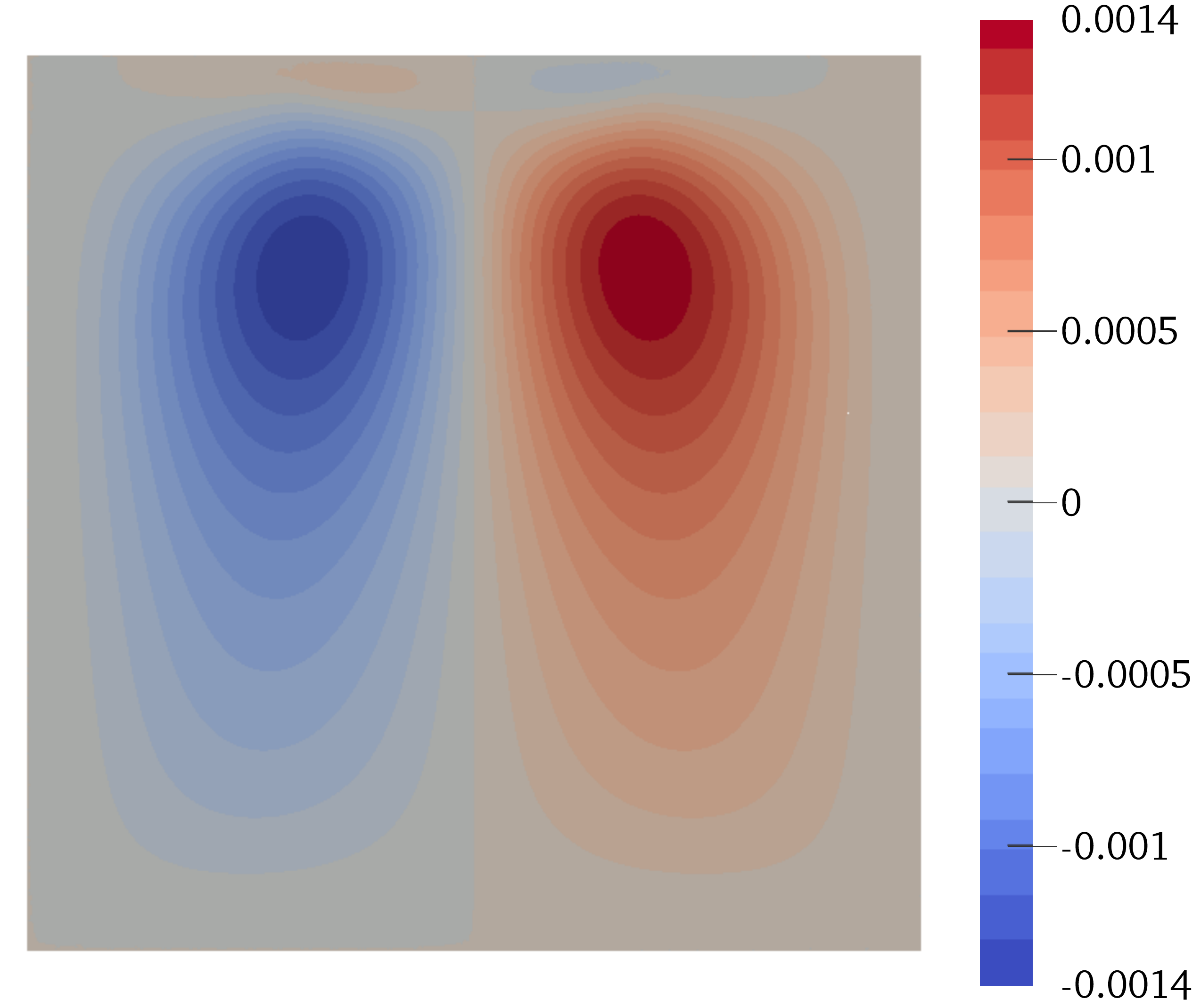}
  \end{minipage}
  \hspace{1mm}
  \begin{minipage}{0.32\textwidth}\centering
    \includegraphics[height=4.5cm]{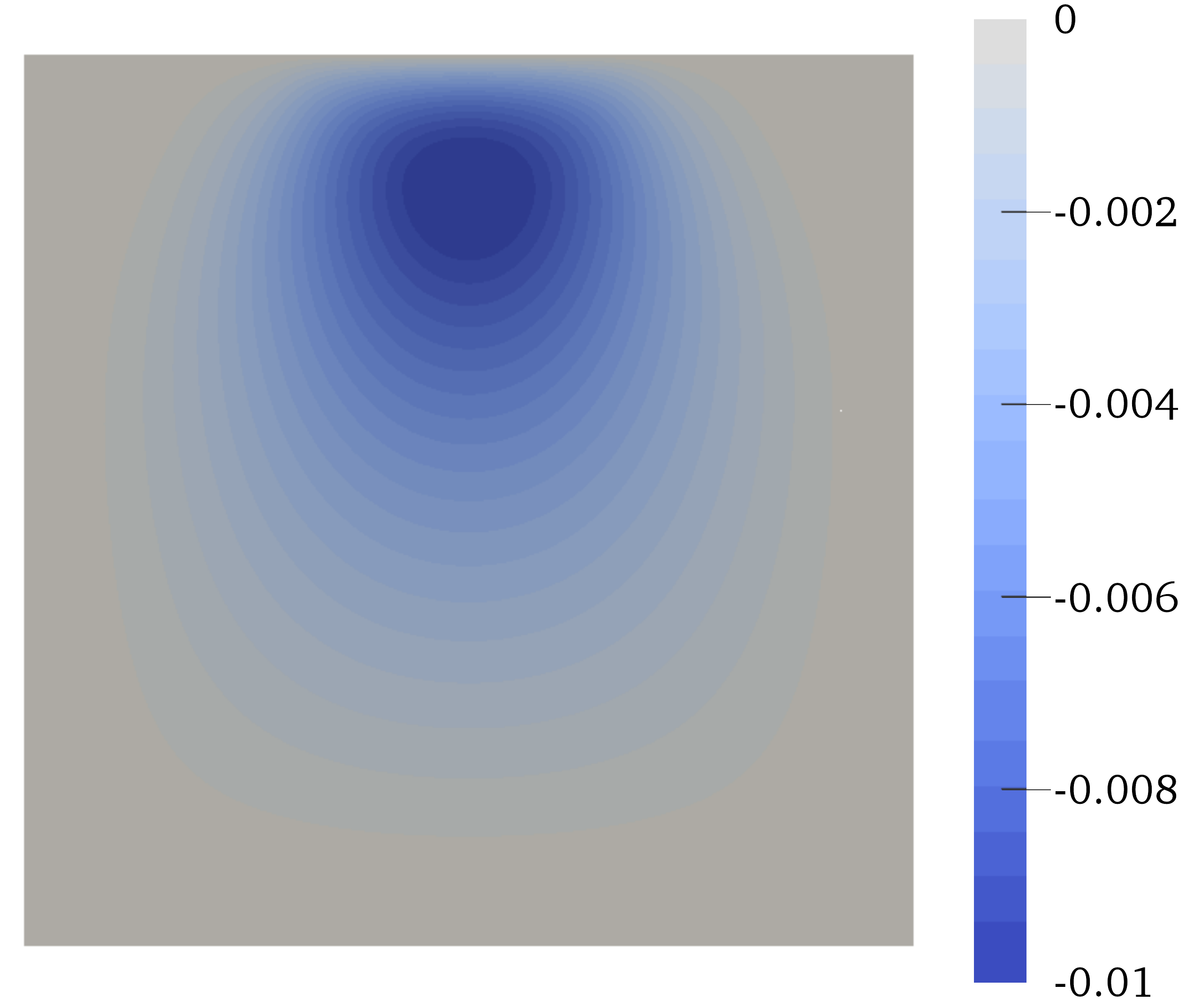}
  \end{minipage}
   \caption{$\Cov({u}_{\mathcal{K},1},{u}_{\mathcal{K},2})\ [{\rm m}^2]$ (left), $\Cov({u}_{\mathcal{K},1},p_{\mathcal{K}})\ [{\rm m}\cdot{\rm kPa}]$ (middle), and $\Cov({u}_{\mathcal{K},2},p_{\mathcal{K}})\ [{\rm m}\cdot{\rm kPa}]$ (right) fields of the footing test case computed at $t=0.2 {\rm s}$ using the PSP method with level $l=5$ and $N_q(l)=2561$.}\label{fig:TC2_Cov}
 \end{figure}
\subsection{Realistic problem}
The last numerical experiment deals with a fluid injection-extraction process and is inspired by~\cite[Section 7]{Kolesov.Vabishchevich.ea:14}. Since we aim to assess the propagation of the uncertainty of input parameters on the solution, we assume that all the coefficients $\mu,\lambda,\alpha$ and $\kappa$ are log-uniformly distributed with the same coefficient of variation ${\rm c_v}(\mu)={\rm c_v}(\lambda)={\rm c_v}(\alpha)={\rm c_v}(\kappa)\sim 0.63$. We set
 \begin{equation}\label{eq:tc3_parameters}
   \begin{aligned}
     \mu(\vec{\xi}) &= 3.75\cdot10^{\frac{\xi_1+1}2} \ {\rm GPa},
     \\
     \lambda(\vec{\xi}) &= 2.5\cdot10^{\frac{\xi_2+1}2} \ {\rm GPa},
     \\
     \alpha(\vec{\xi}) &= 10^{\frac{\xi_3-1}2},
     \\
     \kappa(\vec{\xi}) &= 5\cdot10^{\frac{\xi_4-3}2} \ {\rm Km}^2 {\rm GPa}^{-1} {\rm day}^{-1}.
   \end{aligned}
 \end{equation}
We remark that all the possible realizations of $\alpha$ satisfy~\eqref{eq:bnd.Zimmerman} if $\varphi\le\nicefrac{2}{29}$.
Thus, we consider a medium with reference porosity $\varphi=\nicefrac{2}{29}$ that is filled with water, namely $\Kf=2.2\ {\rm GPa}$. For each realization of the poromechanical coefficients, $c_0$ is again computed according to~\eqref{eq:2order_c0}.
We observe that the average values of $\mu$ and $\lambda$ in~\eqref{eq:tc3_parameters} have the same magnitude as the mechanical coefficients considered in~\cite[Table 1]{Kolesov.Vabishchevich.ea:14} and, assuming that $\mu_{\rm f}=10^{-3}\ {\rm Pa}\cdot \rm{s}$ (water), the average value of $\kappa$ corresponds to an intrinsic permeability of order $1 {\rm D}=10^{-12} {\rm m^2}$. 

\subsubsection{Computational case description}
The rectangular domain $D=[0, 4\, {\rm Km}]\times[0, 1\, {\rm Km}]$ includes two holes representing the injection and extraction wells (see Figure~\ref{fig:TC3_MeshMean}). 
For the sake of simplicity, we impose steady Dirichlet conditions on the holes boundaries, 
$$
\begin{aligned}
  p &= p_1,\quad\text{on }\Gamma_{{\rm d},1},
  \\
  p &= p_2,\quad\text{on }\Gamma_{{\rm d},2},
\end{aligned}
$$
where $\Gamma_{{\rm d},1}$ and $\Gamma_{{\rm d},2}$ denote the left and right hole boundaries, respectively. 
We remark that, due to the linearity of the Biot model, any solution with the conditions above can be reconstructed from a combination of the solutions of two elementary problems: one with $p_1=0, p_2=1$ and another with $p_1=1, p_2=0$.
We treat in this section the case $p_2=-p_1$ and consider the following set of boundary conditions:
$$
\begin{aligned}
  \ms\normal &= \vec0, \;\qquad\qquad \text{on}\ \Gamma_{N}\coloneq\{\vec{x}\in\partial D \st x_2=1\ {\rm Km}\},
  \\
  \vu&=\vec0, \;\qquad\qquad \text{on}\ \Gamma_{D}\coloneq\Gamma_{{\rm d},1}\cup\Gamma_{{\rm d},2}\cup\{\vec{x}\in\partial D \st x_2=0\},
  \\
  \ms\normal\cdot\vec{\tau} &=0, \;\qquad\qquad
  \text{on}\ \partial D\setminus(\Gamma_N\cup\Gamma_D\cup\Gamma_{{\rm d},1}\cup\Gamma_{{\rm d},2}),
  \\
  \vu\cdot\normal&=0, \;\qquad\qquad
  \text{on}\ \partial D\setminus(\Gamma_N\cup\Gamma_D\cup\Gamma_{{\rm d},1}\cup\Gamma_{{\rm d},2}),
  \\
  \kappa\GRAD p &= \vec0, \;\qquad\qquad \text{on}\ \partial D\setminus(\Gamma_{{\rm d},1}\cup\Gamma_{{\rm d},2}),
   \\
   p &= -100\, {\rm kPa}, \quad \text{on}\ \Gamma_{{\rm d},1},
   \\
   p &= 100\, {\rm kPa}, \;\ \,\quad \text{on}\ \Gamma_{{\rm d},2},
\end{aligned}
$$
where $\vec{\tau}$ denotes again the tangent vector on $\partial D$ with arbitrary orientation.
The loading term $\vf$ in~\eqref{eq:SysBiot1}, the fluid source $g$ in~\eqref{eq:SysBiot2}, and the initial condition $\phi_0$ in~\eqref{eq:SysBiot:initial} are all set to zero.
We use $k=1$ in the method of~\cite{Boffi.Botti.Di-Pietro:16} with the Voronoi mesh having $\pgfmathprintnumber{10000}$ elements depicted in Figure~\ref{fig:TC3_MeshMean}. This choice yields a spatial discretization of dimension $\pgfmathprintnumber{151240}$. The steady state of the Biot model is considered achieved at $\tF=1 \ {\rm day}$, simulated with $10$ time steps. The PC approximations are computed using the PSP method with level $l=3$ for a total of $N_q(l)=209$ sparse grid nodes.

\begin{figure}
  \centering
    \includegraphics[scale=0.7]{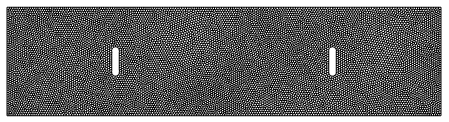}
  \caption{Mesh of the injection-extraction test case. The construction uses the \textsf{PolyMesher} algorithm of~\cite{Talischi:12}.}\label{fig:TC3_MeshMean}
\end{figure}

Figure~\ref{fig:TC3_Mean} shows the mean of the pressure (left) and vertical displacement fields (right). We verify that the mean pressure is equal to deterministic the boundary conditions, namely $\pm 100\, {\rm kPa}$, at the holes boundary. It increases linearly between the two holes (in the $x$-direction) and remains nearly uniform away from the holes. Regarding the displacement, within vertical sections, the vertical component increases from the bottom of the domain (where it is equal to the homogeneous boundary condition) to its maximum value at the free surface. The mean displacement is negative on the left part of the domain (extraction side) and positive on the right part (injection side), consistently with the imposed pressure boundary conditions. Further, the mean displacement vanishes at the vertical centerline of the domain, owing to the symmetry of the problem. 

\begin{figure}
\centering
\includegraphics[scale=0.44]{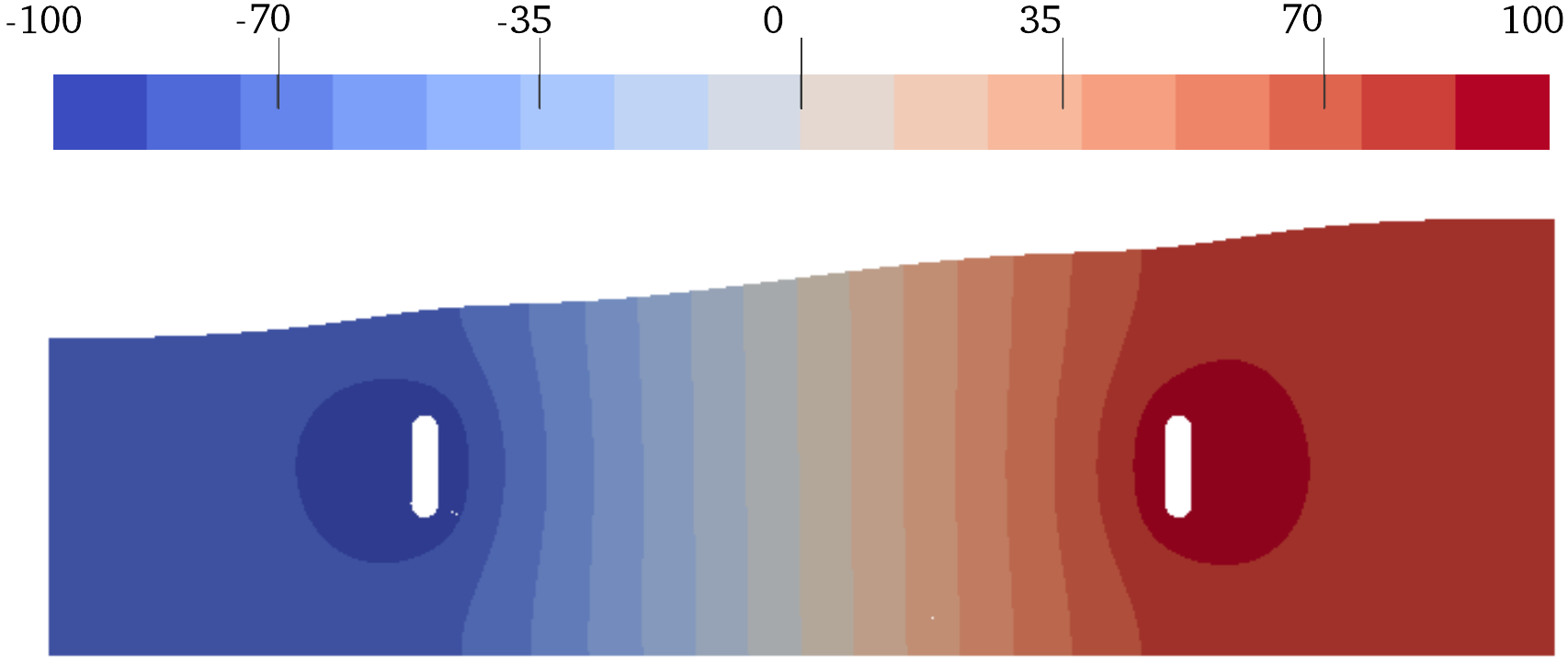}
\hspace{2mm}
\includegraphics[scale=0.44]{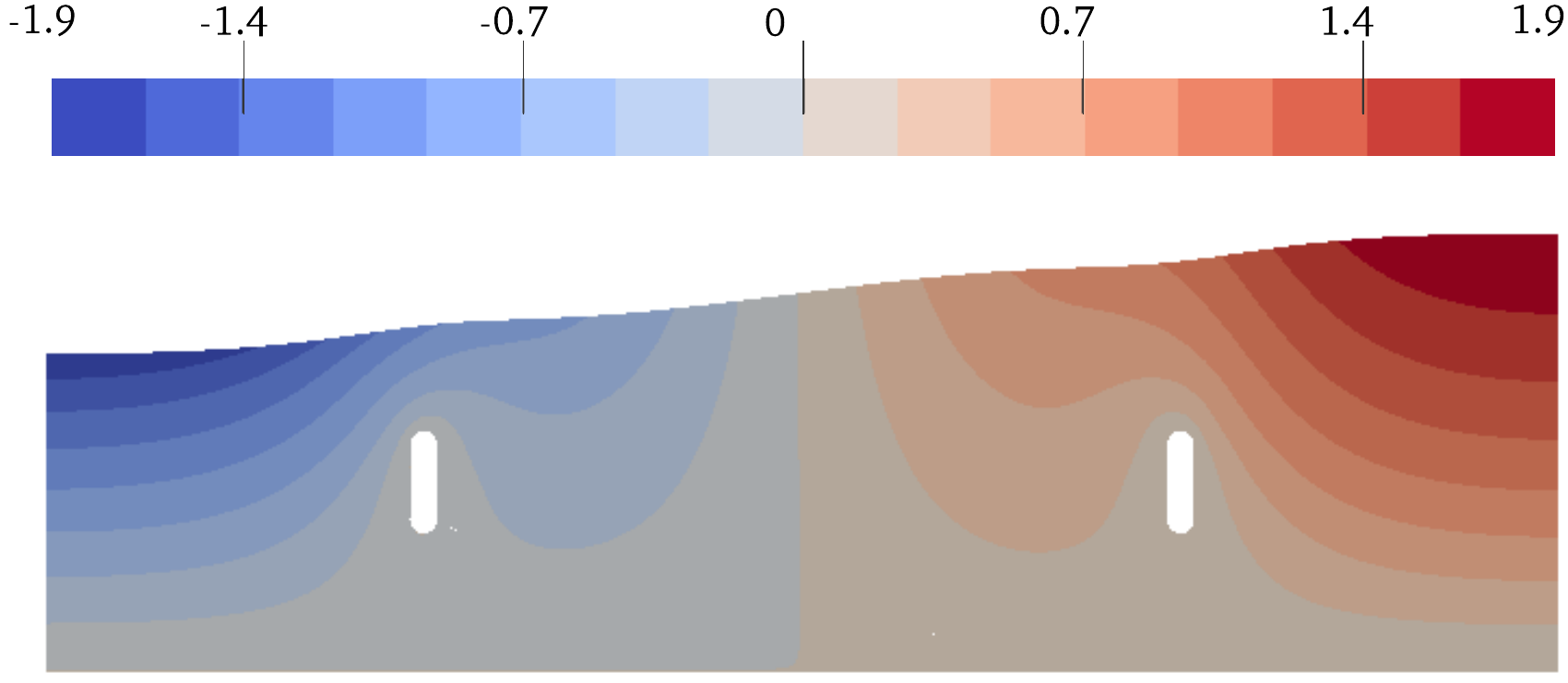}
\caption{Mean pressure  field in ${\rm kPa}$ (left) and mean vertical displacement in ${\rm mm}$ (right) for the injection-extraction test case.
Fields are plotted on the mean domain (deformation scaled by a factor $10^5$) at steady state.}\label{fig:TC3_Mean}
\end{figure} 

Figure~\ref{fig:TC3_Var} shows the variance fields of the Biot solution, on the left column, and the three covariance fields, in the right column.
The pressure and vertical displacement variances are small in between the two holes. 
Indeed, for the considered configuration, the pressure profile at steady state is linear in the $x$-direction and mostly independent of the model coefficients. 
Thus the pressure variance is maximal at the lateral boundaries where the uncertainty in poroelastic coefficients is the most significant.
In contrast, the horizontal displacement variance is the highest between the holes (note that the range of $\Var({u}_{\mathcal{K},1})$ is half that of $\Var({u}_{\mathcal{K},2})$).

The three covariance fields display different features. 
Concerning the correlations $\Cov({u}_{\mathcal{K},1},p_{\mathcal{K}})$ and $\Cov({u}_{\mathcal{K},2},p_{\mathcal{K}})$, it is seen that the later reaches higher magnitudes at the top corners of the domain, as one could have guessed from the variance fields. Further, the pressure and vertical displacement are anti-correlated everywhere in the domain, when the horizontal displacement and pressure have a covariance that alternates sign in the domain.
Finally, consistently with the symmetry of the problem, the covariance between the two displacement components is negative (resp. positive) in the injection (resp. extraction) side of the domain.
\begin{figure}
  \centering
  \begin{minipage}[b]{0.48\textwidth}
    \includegraphics[scale=0.23]{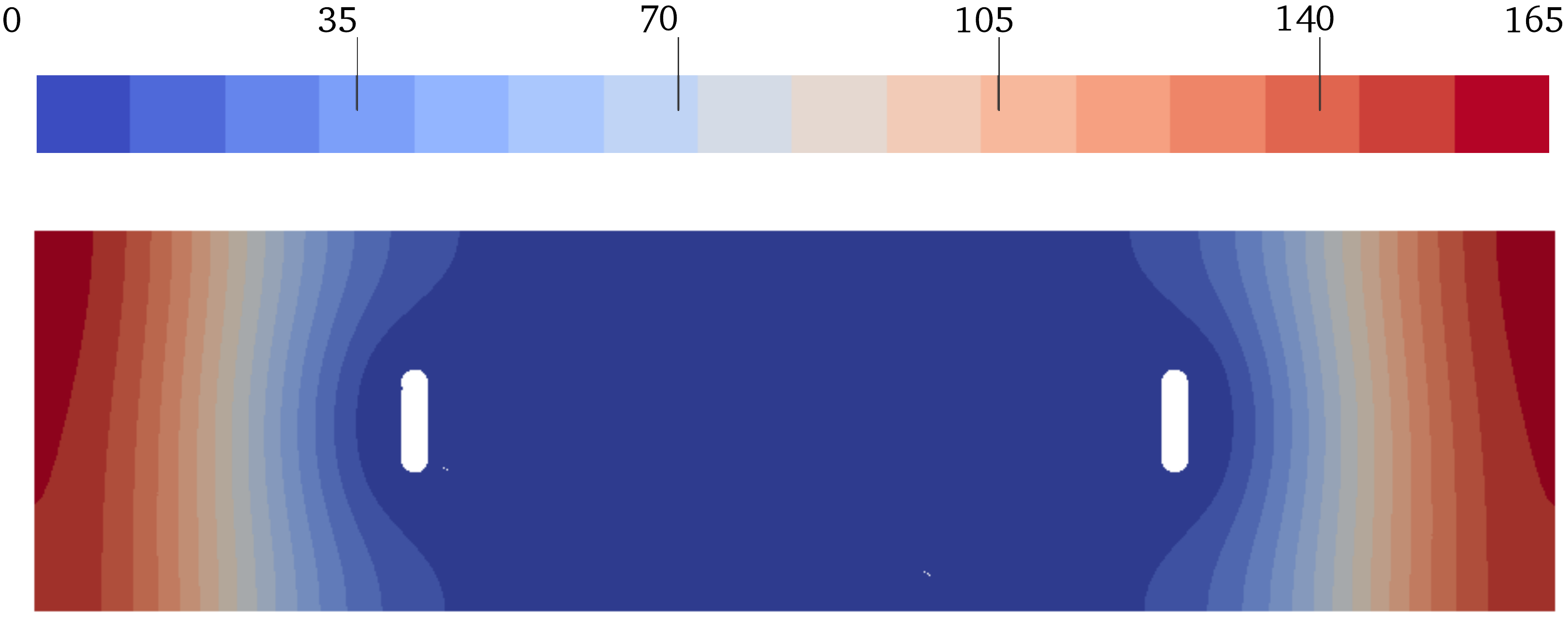}
    \vspace{-2mm}
    \subcaption{$\Var(p_{\mathcal{K}})\ [{\rm kPa}]$}
  \end{minipage}
  \hspace{2mm}
  \begin{minipage}[b]{0.48\textwidth}
    \includegraphics[scale=0.23]{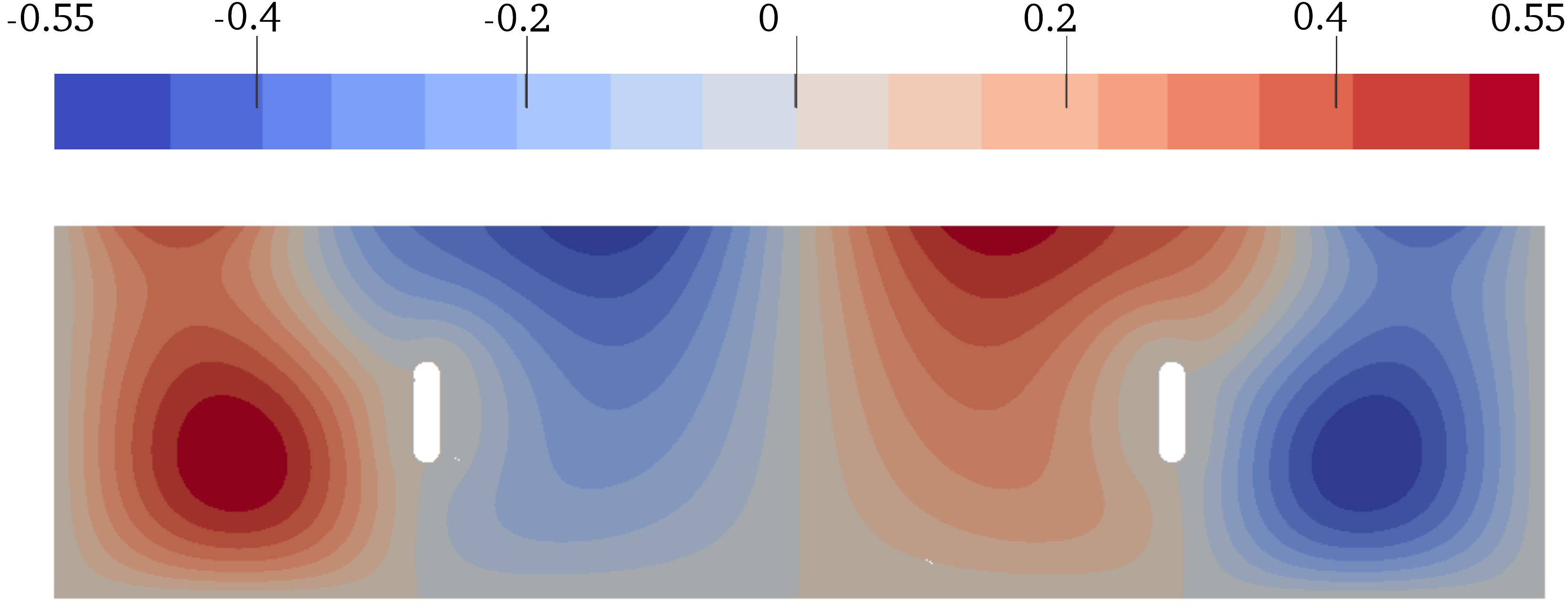}
    \vspace{-2mm}
    \subcaption{$\Cov({u}_{\mathcal{K},1},p_{\mathcal{K}})\ [{\rm mm}\cdot{\rm kPa}]$}
  \end{minipage}
  \vspace{2mm} \\
  \begin{minipage}[b]{0.48\textwidth}
    \includegraphics[scale=0.23]{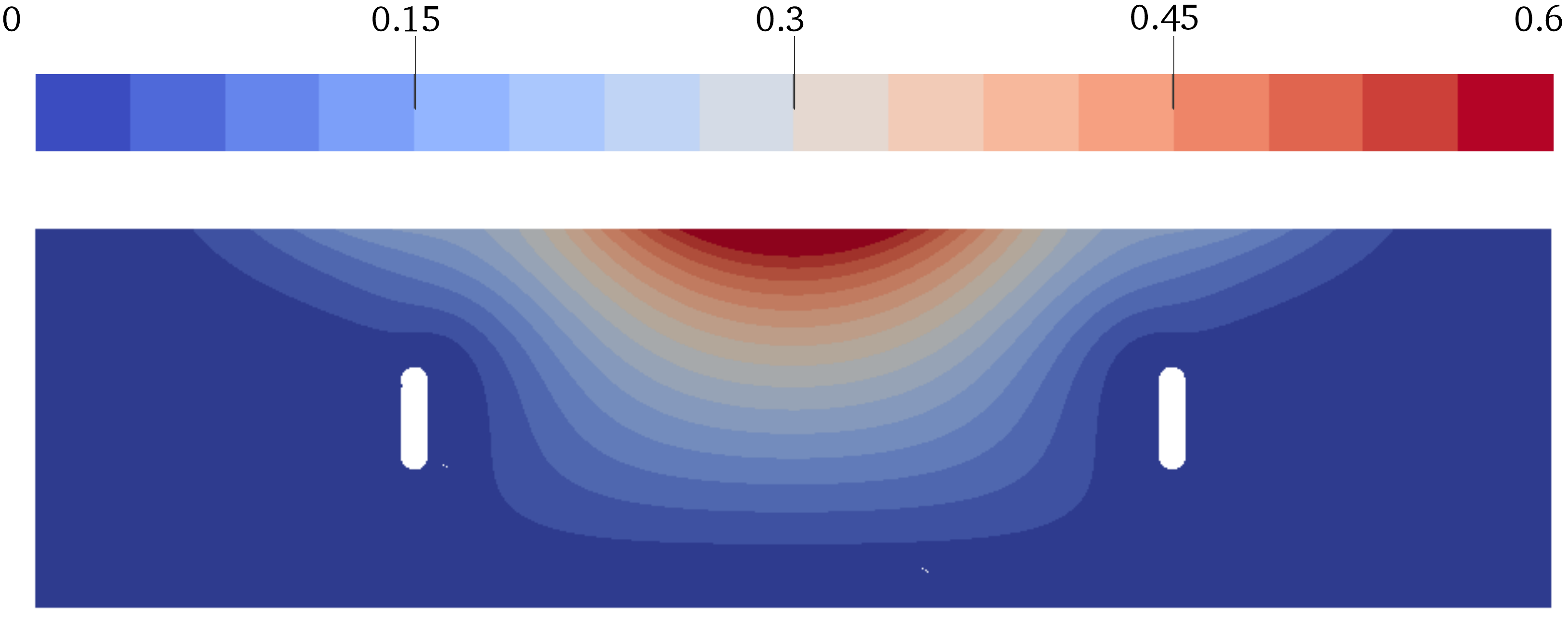}
    \vspace{-2mm}
    \subcaption{$\Var({u}_{\mathcal{K},1})\ [{\rm mm}^2]$}
  \end{minipage}
  \hspace{2mm}
  \begin{minipage}[b]{0.48\textwidth}
    \includegraphics[scale=0.23]{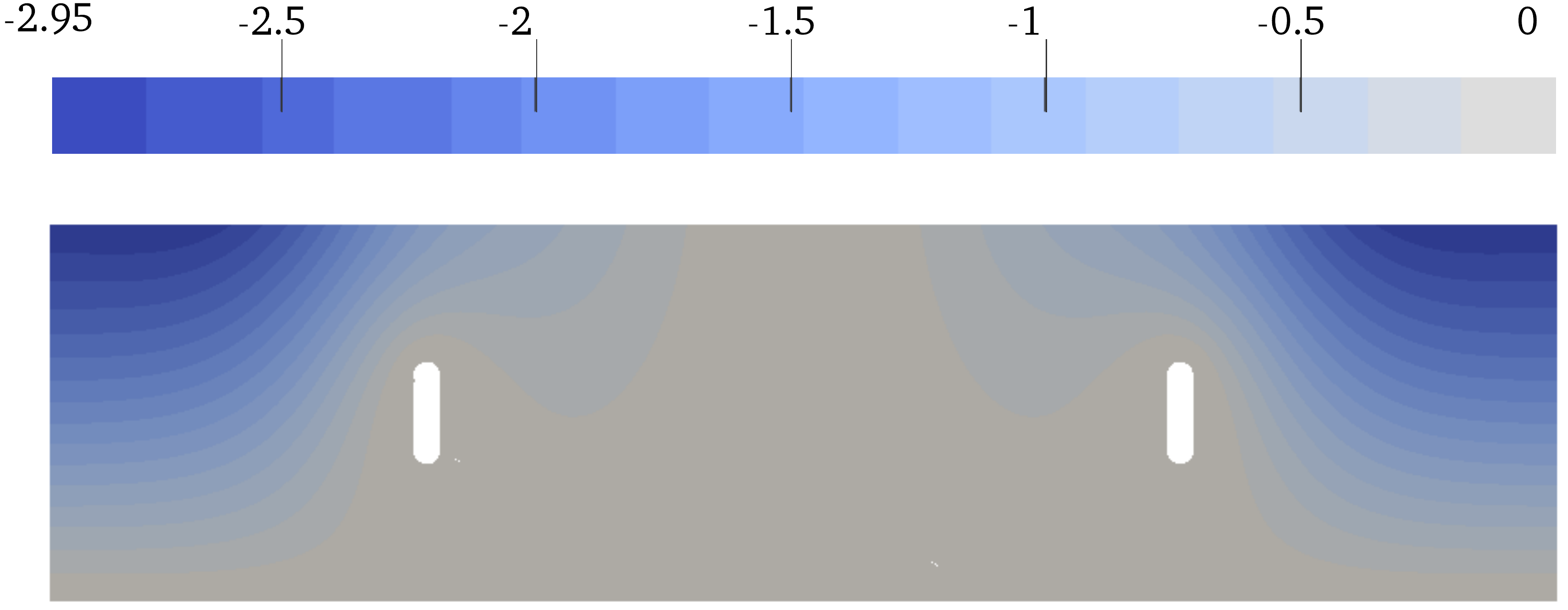}
    \vspace{-2mm}
    \subcaption{$\Cov({u}_{\mathcal{K},2},p_{\mathcal{K}})\ [{\rm mm}\cdot{\rm kPa}]$}
  \end{minipage}
  \vspace{2mm} \\
  \begin{minipage}[b]{0.48\textwidth}
    \includegraphics[scale=0.23]{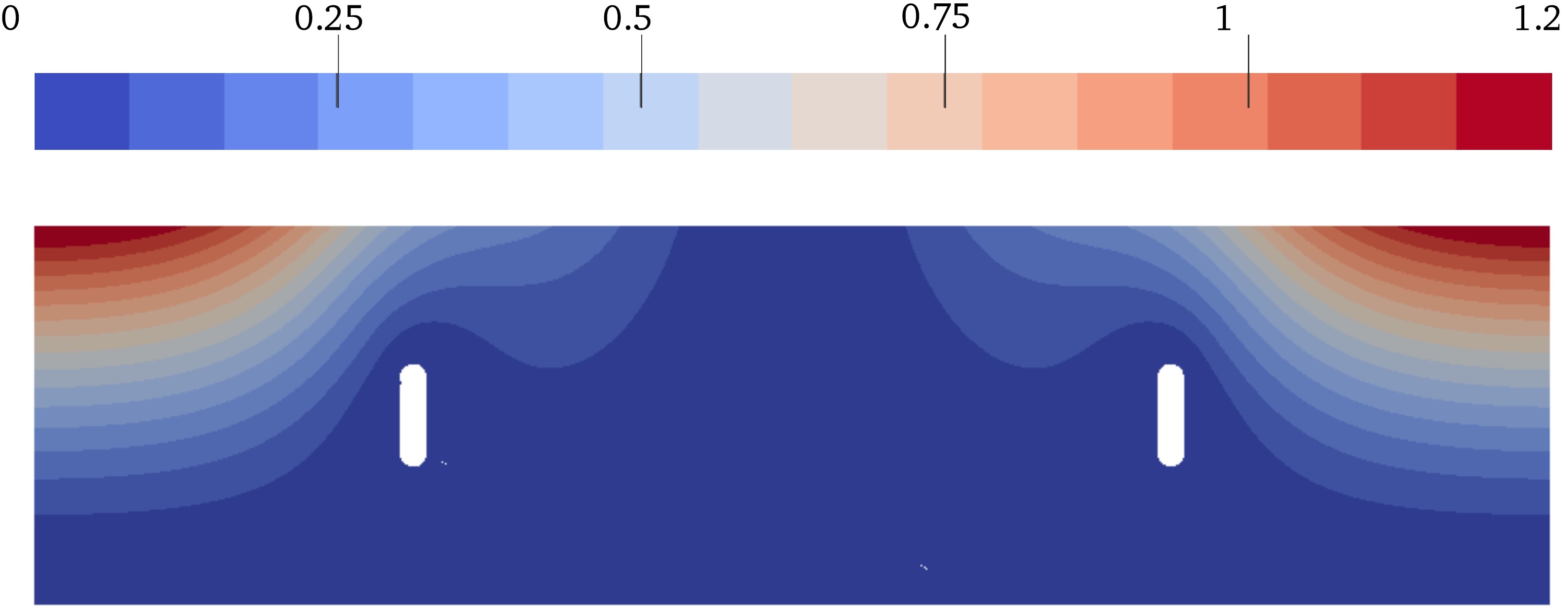}
    \vspace{-2mm}
    \subcaption{$\Var({u}_{\mathcal{K},2})\ [{\rm mm}^2]$}
  \end{minipage}
  \hspace{2mm}
  \begin{minipage}[b]{0.48\textwidth}
    \includegraphics[scale=0.23]{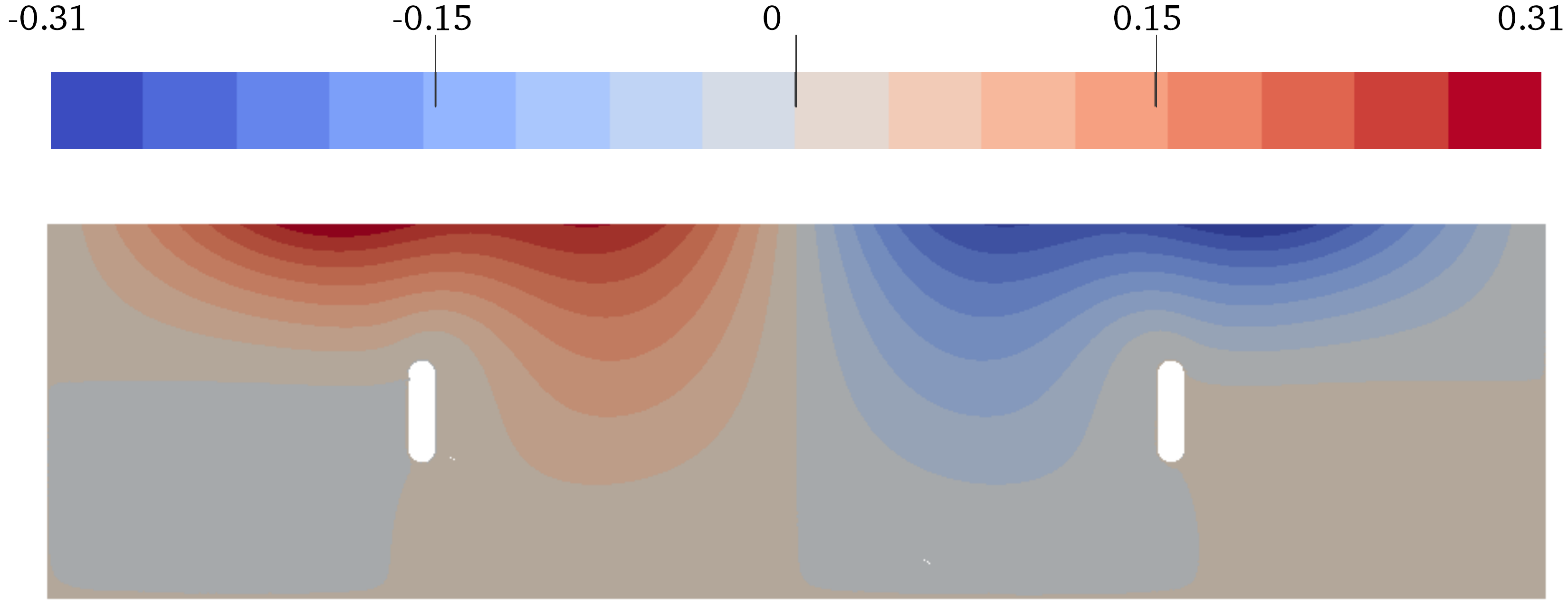}
    \vspace{-2mm}
    \subcaption{$\Cov({u}_{\mathcal{K},1},{u}_{\mathcal{K},2})\ [{\rm mm}^2]$}
  \end{minipage}
  \vspace{2mm} \\
  \caption{Variance (left) and covariance (right) fields for the steady state injection-extraction test case  computed using the PSP method with level $l=3$ and $N_q(l)=209$.}\label{fig:TC3_Var}
\end{figure}

\subsubsection{Sensitivity analysis} 
We conclude the numerical tests by evaluating the different contributions of the input parameters on the variance of the solution. We are mainly interested here in illustrating the information carried by the PC expansion to assess the effect of the random coefficients on the deformation caused by the injection-extraction. 
We perform the sensitivity analysis of the vertical displacement only. Similar analyses can be conducted for the other components of the solution. We compute the first-order partial variances $\Var_i(u_{\mathcal{K},2}(\vec{x},\vec{\xi}))$ defined by
$$
  \Var_i(u_{\mathcal{K},2}) \coloneq \Var(\Esp(u_{\mathcal{K},2}|\xi_i)).
$$
As for the variance, these partial variances can be explicitly derived from the PC expansion of the considered field~\cite{Crestaux.Le-Maitre.ea:09}, and we use here the level $l=3$ sparse grid PSP approximation.
Since each $\xi_i$ in~\eqref{eq:tc3_parameters} parametrizes only one of the poromechanical model coefficients, the first-order variances characterize the contributions to the variance of
the individual poro-mechanical parameters ($\mu(\xi_1)$, $\lambda(\xi_2)$, $\alpha(\xi_3)$ and $\kappa(\xi_4)$) on the displacement. 
We also report the total-order variance associated to the $\xi_i$, denoted $\Var_{T_i}(u_{\mathcal{K},2})$, which are defined by
$$
  \Var_{T_i}(u_{\mathcal{K},2}) \coloneq \Var(u_{\mathcal{K},2})-
  \Var(\Esp(u_{\mathcal{K},2}|\vec{\xi}_{\setminus i})),
$$
where we have denoted $\vec{\xi}_{\setminus i}=\left( \xi_{j\ne i}\right)$ the triplet containing all the random variables but $\xi_i$.
The total-order variance $\Var_{T_i}$ measure the variability induced by $\xi_i$ (that is the poromechanical coefficient it parametrizes) including all its interactions with other uncertainty sources.

In Figure~\ref{fig:VarComp_TC3} we compare the total variance $\Var(u_{\mathcal{K},2})$ with the sum of the first-order partial variances $\sum_{i=1}^4\Var_i(u_{\mathcal{K},2})$ of the vertical displacement. The difference between these two fields amounts to the contribution of higher order partial variances, namely, the part of the variance incurring to interactions between random coefficients. Here, these interactions are seen to be small but not negligible. We expect coupled effects between uncertain parameters in the present uncertainty model, because of the dependence of $c_0$ on the three first coefficients (see~\eqref{eq:2order_c0}).  
In Figure~\ref{fig:sens_TC3} we plot the first-order (left half) and total-order partial variances (right half) related to each input parameter. 
We observe that the parameter with the main influence on the vertical displacement is the shear modulus $\mu$ (parametrized by $\xi_1$), while the effects of the dilatation modulus $\lambda$ (parametrized by $\xi_2$) and hydraulic mobility $\kappa$ (parametrized by $\xi_4$) are much weaker. In particular, we notice that $\kappa$ has almost no influence on the solution variability. Finally, we remark that the effect of the uncertain Biot-Willis coefficient $\alpha$ on the vertical displacement at the surface is less important than the one associated to  $\mu$, but considerably larger than the one related to $\lambda$.
\begin{figure}
   \centering
   \includegraphics[scale=0.76]{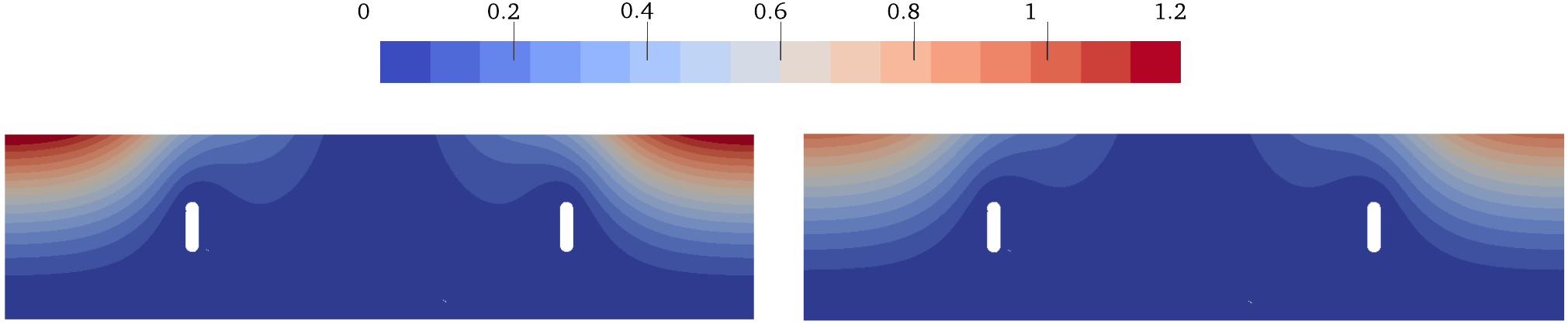}
   \caption{Comparison between the total variance $\Var(u_{\mathcal{K},2})$ and 
   the sum of the first-order partial variances $\sum_i \Var_i(u_{\mathcal{K},2})$.
   The variances are reported in ${\rm mm}^2$. 
   \label{fig:VarComp_TC3}}
\end{figure} 
\begin{figure}
  \centering
    \includegraphics[scale=0.74]{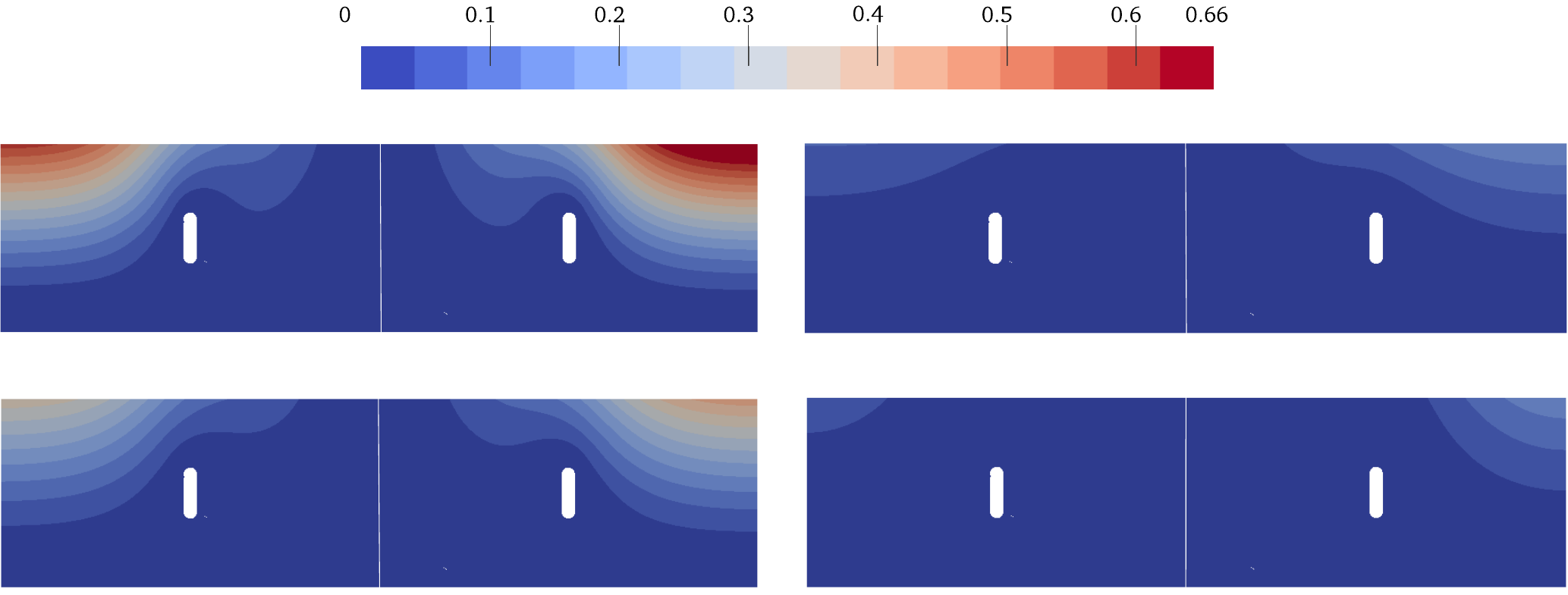}
    \caption{First and total-order partial variances of the vertical displacement (in ${\rm mm}^2$), associated to $\mu(\xi_1)$ (top left), $\lambda(\xi_2)$ (top right), $\alpha(\xi_3)$ (bottom left) and $\kappa(\xi_4)$ (bottom right). 
    The first-order (resp. total) partial variance is plotted in the left (resp. right) half of each plot.
  \label{fig:sens_TC3}} 
\end{figure}


\section{Conclusion}

This paper contains a theoretical and numerical study of the poroelasticity problem with random coefficients. It contains three major contributions.
The first one is a study of the dependence between the physical coefficients, leading to an uncertainty model where the specific storage coefficient is expressed in terms of the other coefficients. 
This construction ensures that the resulting problem remains almost surely within the set of physically possible configurations.
This uncertainty model also results in a reduction in the number of independent coefficients for the stochastic problem to be solved. 
The second contribution concerns the well-posedness of the stochastic Biot problem; we prove the existence of a second order solution provided for random model coefficients satisfying some weak mathematical assumptions that we detail. The existence of a second order solution appears to be a new mathematical result.
The third contribution is a discretization of the stochastic Biot problem using a (non-intrusive) sparse pseudo spectral projection method. This approach relies on an ensemble of deterministic simulations carried out using a Hybrid High-Order discretization. The proposed method is further tested on a set of model problems.


\section*{Acknowledgements}
This work was partially funded by the Bureau de Recherches G\'{e}ologiques et Mini\`{e}res. The work of M. Botti was additionally partially supported by Labex NUMEV (ANR-10-LABX-20) ref. 2014-2-006.

\bibliographystyle{plain}
\bibliography{Main}

\end{document}